\documentclass[twoside,a4paper,12pt]{article}
\usepackage{amsmath,amsthm,amsfonts,amssymb}
\usepackage{mathrsfs}
\usepackage{color,graphicx}

\makeatletter
\def\@seccntformat#1{\csname the#1\endcsname.\quad}
\renewcommand\section{\@startsection {section}{1}{\z@}%
                                   {-3.5ex \@plus -1ex \@minus -.2ex}%
                                   {2.3ex \@plus.2ex}%
                                   {\normalfont\bf\center }}
\renewcommand\subsection{\@startsection {subsection}{1}{\z@}%
                                   {-3.5ex \@plus -1ex \@minus -.2ex}%
                                   {2.3ex \@plus.2ex}%
                                   {\normalfont\bf}}
\makeatother

\newcommand{\res}{\operatorname{Res}}
\newcommand{\RE}{\operatorname{Re}}
\newcommand{\IM}{\operatorname{Im}}
\newtheoremstyle{new-thm}
 {3pt}
 {3pt}
 {\it}
 {0pt} 
 {\bf}
 {.}
 {.5em}
 {}
\newtheoremstyle{new-def}
 {3pt}
 {3pt}
 {\rm}
 {0pt} 
 {\bf}
 {.}
 {.5em}
 {}

\theoremstyle{new-thm}
  \newtheorem{thm}{Theorem}
  
  \newtheorem{lemma}[thm]{Lemma}
  \newtheorem{prop}[thm]{Proposition}
  
\theoremstyle{new-def}

  \newtheorem{remark}[thm]{Remark}

\numberwithin{equation}{section}
\numberwithin{thm}{section}

\begin{document}

\vspace*{2cm} 

\begin{center}
{\large\bf Hitting times of Bessel processes, 
volume of Wiener sausages
and zeros of Macdonald functions}
\end{center}

\bigskip

\begin{center}
{\large Yuji Hamana and Hiroyuki Matsumoto}
\footnote{This work is partially supported by the Grant-in-Aid
for Scientific Research (C) No.23540183 and No.24540181
of Japan Society for the Promotion of Science (JSPS).}
\end{center}


\bigskip


\begin{abstract}
We derive formulae for some ratios of the Macdonald functions, 
which are simpler and easier to treat than known formulae.
The result gives two applications in probability theory.
One is the formula for the L{\'e}vy measure of the distribution
of the first hitting time of a Bessel process and the other is
an explicit form for the expected volume of the Wiener sausage
for an even dimensional Brownian motion.
Moreover, the result enables us to write down the algebraic equations
whose roots are the zeros of Macdonald functions.
\end{abstract}



\section{Introduction}

\noindent
The (modified) Bessel functions appear in various kinds of situations.
In probability theory, for example, the modified Bessel functions
of the first kind, denoed by $I_\nu$, appear in the explicit form
for the transition probability densities of the Bessel processes.
In this article we are concerned with the ratio of the modified
Bessel functions. It is known that such functions represent
the Laplace transforms of the first hitting times of
the Bessel processes (cf. \cite{2,12}) and of
the expectations of the Wiener sausage (cf. \cite{5}).  
From an analytical point of view Ismail~et~al~\cite{9,10}
have studied when such functions are completely monotone.

We are mainly concerned with the ratios of the modified Bessel function
of the second kind $K_\nu$, so-called the Macdonald functions.
Ismail~\cite{9} has shown that $(K_{\nu+1}/K_\nu)(\sqrt{z})$
is completely monotone by expressing it as a Stieltjes transform of
some function. From his expression we can invert the Laplace transform,
but the resulting formula seems to be complicated.
The purpose of this article is to rewrite the resulting ratio
in a simpler form by means of the zeros
of $K_\nu$ and to invert the Laplace transform, which completes
a partial result in \cite{21}.
The result is applied to two questions in probability theory
and a study on the zeros of $K_\nu$.

Recently, in connection with the first hitting times of Bessel processes,
the authors~\cite{7} have studied another type of the ratios
of the Macdonald functions and decomposed it into a sum of
several functions which are easy to treat.
A similar method via some contour integrals is effective in this article.

The purpose of \cite{7} is to show an explicit form of the distribution
function for the first hitting time $\tau_{a,b}^{(\nu)}$ to $b$ 
of the Bessel process with index $\nu$ starting at $a$. The density for
$\tau_{a,b}^{(\nu)}$ and its asymptotics have been discussed in \cite{8}.
The infinite divisibility of the distribution was first investigated 
by Kent~\cite{12}. To be accurate, the conditional distribution of
$\tau_{a,b}^{(\nu)}$ under the condition that it is finite is
infinitely divisible (cf. \cite{10}).
General theory in the infinite divisibility of the distributions 
of the first hitting times of one-dimensional diffusion processes 
is given by Yamazato~\cite{20}.

As referred in \cite{7}, the function $K_{\nu+1}/K_\nu$ appear 
when we give an expression for the L\'evy measure.   
We may apply our result on the ratio of the Macdonald functions 
to obtain an explicit expression for the L\'evy measure 
of the distribution of $\tau_{a,b}^{(\nu)}$.

Moreover, the function $(K_{d/2}/K_{d/2-1})(r\sqrt{2\lambda})$
in $\lambda>0$ represents the Laplace transform of the expectation
of the Wiener sausage for the $d$-dimensional Brownian motion
associated with a close ball with radius $r$ (cf. \cite{5}).  
In the case when $d$ is odd, Hamana~\cite{6} divided
the function into the sum of several functions
of which the inverse Laplace transforms can be obtained easily
and deduced an exact form of the mean volume of
the Wiener sausage by means of zeros of $K_{d/2-1}$.
By using our result we can show that, also in the even dimensional case, 
the expectation is represented in a similar form.
We should remark that the Wiener sausage for a Brownian motion
associated with a general compact set is investigated in
\cite{14,18} and so on, and that the same problem
for a stable sausage is discussed in \cite{1,16}.

In the results mentioned so far we have express several quantities
by using the zeros of $K_\nu$. If we consider the asymptotic behavior
of the ratio $(K_{\nu+1}/K_\nu)(w)$ as $w\to\infty$ or $w\to0$,
we may conversely obtain information on the zeros of $K_\nu$.
In particular, we see that the zeros are the roots
of some algebraic equations with real coefficients
and show the way to obtain the coefficients.
The result is the improvement of that on $K_n$
for an integer $n$ given in \cite{21}.

This article is organized as follows. Section 2 is devoted to
a decomposition of the functions $K_{\nu+1}/K_\nu$.
Section 3 is devoted to a representation
of the L{\'e}vy measure of the first hitting time of
the Bessel process. We calculate the expected volume
of the Wiener sausage for
the even dimensional Brownian motion in Section 4
and discuss its large time asymptotics in Section 5.
In the final Section 6 we study the complex zeros
of the Macdonald functions.


\section{Ratios of Macdonald functions}

\noindent
For each complex number $\nu$ the modified Bessel function
of order $\nu$ is the fundamental solutions of
the modified Bessel differential equation
\begin{equation}
z^2\frac{d^2w}{dz^2}+z\frac{dw}{dz}-(z^2+\nu^2)w=0.
\label{2.1}
\end{equation}
The standard notation $I_\nu$ and $K_\nu$ are used to denote
the functions, which are called the first kind and
the second kind, respectively. See \cite{3,13,19} et al.
Especially, $K_\nu$ is also called the Macdonald function
of order $\nu$. In this article we treat only the case
when the all orders of modified Bessel functions are real.

Before giving the result, we recall several facts
on the zeros of the Macdonald function. For $\nu\in\mathbb R$ let $N(\nu)$
be the number of zeros of $K_\nu$. It is known
that $N(\nu)$ is equal to $|\nu|-1/2$ if $\nu-1/2$ is an integer
and that $N(\nu)$ is the even number closest to $|\nu|-1/2$ otherwise.
$N(\nu)=0$ if $|\nu|<3/2$. If $|\nu|=2n+3/2$ for some integer $n$,
$z^{|\nu|}e^z K_\nu(z)$ is a polynomial of degree $2n+1$
and has a real (negative) zero. Otherwise, $K_\nu$ does not have
real zeros. Each zero, if exists, lies in the half plain
$\{z\in{\mathbb C}\,;\,\RE(z)<0\}$, denoted by ${\mathbb C}^-$.
When $N(\nu)\geqq1$, we write $z_{\nu,1},z_{\nu,2},\dots,z_{\nu,N(\nu)}$
for the zeros. Since $K_\nu$ is one of the fundamental solutions of
the second order equation \eqref{2.1}, all zeros of $K_\nu$ are of multiplicity one
by the uniqueness of the solution of ordinary differential equations.
This means that all zeros of $K_\nu$ are distinct.
For details, see \cite[pp.511--513]{19}.

Let $D=\{z\in\mathbb C\setminus\{0\}\,;\,|\arg z|<\pi\}$
and $D_\nu=\{z\in D\,;\,K_\nu(z)\neq0\}$ for $\nu\in\mathbb R$.
The purpose of this section is to show the following theorem.

\medskip

\begin{thm}\label{Theorem 2.1}
Let $w\in D_\nu$ and $\nu^+=\max\{\nu,0\}$.
In addition, we put, for $\mu\geqq0$,
\begin{equation}
G_\mu(x)=K_\mu(x)^2+\pi^2 I_\mu(x)^2+2\pi\sin(\pi\mu)K_\mu(x)I_\mu(x),
\quad x>0.
\label{2.2}
\end{equation}
{\rm (1)} When $\nu-1/2$ is an integer, we have that,
if $|\nu|=1/2$,
\begin{equation}
\frac{K_{\nu+1}(w)}{K_\nu(w)}=1+\frac{2\nu^+}w
\label{2.3}
\end{equation}
and that, if $|\nu|\geqq3/2$,
\begin{equation}
\frac{K_{\nu+1}(w)}{K_\nu(w)}=1+\frac{2\nu^+}w
+\sum_{j=1}^{N(\nu)}\frac1{z_{\nu,j}-w}.
\label{2.4}
\end{equation}
{\rm (2)} When $\nu-1/2$ is not an integer,
we have that, if $|\nu|<3/2$,
\begin{equation}
\frac{K_{\nu+1}(w)}{K_\nu(w)}=1+\frac{2\nu^+}w
+\cos(\pi\nu)\int_0^\infty \frac{dx}{x(x+w)G_{|\nu|}(x)}
\label{2.5}
\end{equation}
and that, if $|\nu|>3/2$,
\begin{equation}
\frac{K_{\nu+1}(w)}{K_\nu(w)}=1+\frac{2\nu^+}w
+\sum_{j=1}^{N(\nu)}\frac1{z_{\nu,j}-w}
+\cos(\pi\nu)\int_0^\infty \frac{dx}{x(x+w)G_{|\nu|}(x)}.
\label{2.6}
\end{equation}
\end{thm}

\medskip

We note that this theorem has been established in \cite{21}
when $\nu$ is an integer.

For a proof of Theorem \ref{Theorem 2.1}, it is sufficient to consider
the case of $\nu\geqq0$ because of the formula $K_\mu=K_{-\mu}$
and the recurrence relation
\begin{equation}
K_{\mu+1}(z)-K_{\mu-1}(z)=\frac{2\mu}z K_\mu(z)
\label{2.7}
\end{equation}
(cf. \cite[p.79]{19}). In fact, we see for $\nu<0$
\begin{equation}
\frac{K_{\nu+1}(z)}{K_\nu(z)}=
\frac{K_{|\nu|+1}(z)}{K_{|\nu|}(z)}+\frac{2\nu}z,
\quad z\in D_\nu=D_{|\nu|}.
\label{2.8}
\end{equation}

Formula \eqref{2.3} is easily obtained from
\[
K_{1/2}(z)=\sqrt{\frac\pi{2z}}e^{-z},\quad
K_{3/2}(z)=\sqrt{\frac\pi{2z}}e^{-z}\frac{z+1}z.
\]
Moreover \eqref{2.4} has been already established in \cite{6}.
Therefore we concentrate on the case
when $\nu-1/2$ is not an integer.

To prove \eqref{2.5} and \eqref{2.5}, we need three lemmas.
One is an uniform estimate of the Macdonald function,
which has been proved in \cite{7}.

\medskip

\begin{lemma}\label{Lemma 2.2}
Let $\delta\in(0,3\pi/2)$ be given.
For $\mu\geqq0$ we have
\begin{equation}
K_\mu(z)=\sqrt{\frac\pi{2z}}e^{-z}\{1+E_\mu(z)\}
\label{2.9}
\end{equation}
if $|\arg z|\leqq 3\pi/2-\delta$.
Here $|E_\mu(z)|\leqq A_\mu/|z|$ for a constant $A_\mu$
which is independent of $z$.
\end{lemma}

\medskip

The other two give asymptotic behavior on the real line
of the functions involving the modified Bessel functions.
Both of them are easily shown by the formula
\begin{equation}
I_\mu(x)=\frac{e^x}{\sqrt{2\pi x}}
\biggl\{1-\frac{4\mu^2-1}{8x}+o\biggl(\frac1x\biggr)\biggr\}
=\frac{e^x}{\sqrt{2\pi x}}\{1+o(1)\},\, x\to\infty
\label{2.10}
\end{equation}
for $\mu\geqq0$ (cf. \cite[p.203]{19}) and we omit the detailed proofs.

\medskip

\begin{lemma}\label{Lemma 2.3}
Let $\zeta, \eta,\xi\geqq0$.
It follows that, as $x\to\infty$,
\[
\begin{split}
&G_\zeta(x)=\frac{\pi e^{2x}}{2x}\{1+o(1)\},\quad
\frac{K_\eta(x)K_\xi(x)}{G_\zeta(x)}=e^{-4x}\{1+o(1)\},\\
&\frac{I_\eta(x)I_\xi(x)}{G_\zeta(x)}=\frac1{\pi^2}\{1+o(1)\},\quad
\frac{I_\eta(x)K_\xi(x)}{G_\zeta(x)}=\frac{e^{-2x}}\pi\{1+o(1)\}.
\end{split}
\]
\end{lemma}


\begin{lemma}\label{Lemma 2.4}
Let $\mu\neq0$. It follows that, as $x\to\infty$,
\begin{equation}
I_\mu(x)-I_{\mu+1}(x)=\frac1{\sqrt{2\pi x}}e^x
\biggl\{\frac{2\mu+1}{2x}+o\biggl(\frac1x\biggr)\biggr\}.
\label{2.11}
\end{equation}
\end{lemma}

\medskip

We are now ready to show Theorem \ref{Theorem 2.1}.
We only consider the case when $\nu-1/2$ is not an integer.
Let $\alpha\in(0,1)$ and $w\in D_\nu$.
For $z\in D_\nu$ with $z\neq w$, we set
\[
f_{\nu,\alpha}^w(z)=\frac1{z^\alpha (z-w)}
\frac{K_{\nu+1}(z)}{K_\nu(z)}.
\]

Letting $\varepsilon$ and $R$ be positive numbers with $\varepsilon<1$ and
$\varepsilon<R$ and setting
\[
\theta_{R,\varepsilon}=\operatorname{Arcsin}\frac\varepsilon R
\in\biggl(0,\frac\pi2\biggr),
\]
we consider the same piecewise $C^1$-curve $\gamma$ as in \cite{7}
defined by
\[
\begin{split}
&\gamma_0\,:\,z=Re^{i\theta},
\,\,-\pi+\theta_{R,\varepsilon}\leqq\theta\leqq\pi-\theta_{R,\varepsilon},\\
&\gamma_1\,:\,z=x+i\varepsilon,
\,\,-R\cos\theta_{R,\varepsilon}\leqq x\leqq0\\
&\gamma_2\,:\,z=x-i\varepsilon,
\,\,-R\cos\theta_{R,\varepsilon}\leqq x\leqq0\\
&\gamma_3\,:\,z=\varepsilon e^{i\theta},
\,\,-\pi/2\leqq\theta\leqq\pi/2,\\
&\gamma=\gamma_0+\gamma_1-\gamma_3-\gamma_2.
\end{split}
\]
We take $R$ so large and $\varepsilon$ so small
that $w$ and all zeros of $K_\nu$ are inside $\gamma$.
Then, setting
\[
\varPsi(R,\alpha,\varepsilon)=\frac1{2\pi i}
\int_\gamma f_{\nu,\alpha}^w(z)dz,\quad
\varPsi_k(R,\alpha,\varepsilon)=\frac1{2\pi i}\int_{\gamma_k}
f_{\nu,\alpha}^w(z)dz
\]
for $k=0,1,2$ and
\[
\varPsi_3(\alpha,\varepsilon)=\frac1{2\pi i}\int_{\gamma_3}
f_{\nu,\alpha}^w(z)dz,
\]
we have
\begin{equation}
\varPsi(R,\alpha,\varepsilon)=\varPsi_0(R,\alpha,\varepsilon)
+\varPsi_1(R,\alpha,\varepsilon)-\varPsi_2(R,\alpha,\varepsilon)
-\varPsi_3(\alpha,\varepsilon).
\label{2.12}
\end{equation}

The singular points of $f_{\nu,\alpha}^w$ inside $\gamma$
are $w$ and the zeros of $K_\nu$, which are all poles of order one.
Hence the residue theorem yields
\begin{equation}
\varPsi(R,\alpha,\varepsilon)=
\begin{cases}
\res(w;f_{\nu,\alpha}^w)\quad&\text{if $N(\nu)=0$,}\\
\displaystyle\res(w;f_{\nu,\alpha}^w)
+\sum_{j=1}^{N(\nu)}\res(z_{\nu,j};f_{\nu,\alpha}^w)
\quad&\text{if $N(\nu)\geqq1$.}
\end{cases}
\label{2.13}
\end{equation}
Here $\res(v;f)$ is the residue of a function $f$
at a pole $v$. It is obvious that
\[
\res(w;f_{\nu,\alpha}^w)=\frac1{w^\alpha}\frac{K_{\nu+1}(w)}{K_\nu(w)}.
\]
When $N(\nu)\geqq1$, by the formula $zK_\nu'(z)-\nu K_\nu(z)=-zK_{\nu+1}(z)$
(cf. \cite[p.29]{19}), we have
\[
\res(z_{\nu,j};f_{\nu,\alpha}^w)=
\frac1{z_{\nu,j}^\alpha(z_{\nu,j}-w)}
\frac{K_{\nu+1}(z_{\nu,j})}{K_\nu'(z_{\nu,j})}
=-\frac1{z_{\nu,j}^\alpha(z_{\nu,j}-w)}.
\]
Hence we obtain from \eqref{2.13}
\begin{equation}
\lim_{\varepsilon\to0}\lim_{\alpha\to0}\lim_{R\to\infty}
\varPsi(R,\alpha,\varepsilon)=
\begin{cases}
\dfrac{K_{\nu+1}(w)}{K_\nu(w)}&\text{if $N(\nu)=0$,}\\
\displaystyle \dfrac{K_{\nu+1}(w)}{K_\nu(w)}
-\sum_{j=1}^{N(\nu)}\frac1{z_{\nu,j}-w}
&\text{if $N(\nu)\geqq1$}
\end{cases}
\label{2.14}
\end{equation}
if we show that the limit on the left hand side exists.

We fix $\varepsilon>0$, $\alpha>0$ and consider the right hand side of
\eqref{2.12}. By \eqref{2.9} we have for sufficiently large $R$
\[
|\varPsi_0(R,\alpha,\varepsilon)|\leqq\frac R{2\pi}
\int_{-\pi+\theta_{R,\varepsilon}}^{\pi-\theta_{R,\varepsilon}}
|f_{\nu,\alpha}^w(Re^{i\theta})|d\theta
\leqq\frac1{R^\alpha}\frac R{R-|w|}\frac{1+A_{\nu+1}/R}{1-A_\nu/R},
\]
which immediately yields that
\[
\lim_{R\to\infty}\varPsi_0(R,\alpha,\varepsilon)=0.
\]

For the integral $\varPsi_1(R,\alpha,\varepsilon)$, we have
\[
\begin{split}
\varPsi_1(R,\alpha,\varepsilon)
&=\frac1{2\pi i}\int_{-R\cos\theta_{R,\varepsilon}}^0
\frac1{(x+i\varepsilon)^\alpha(x+i\varepsilon-w)}
\frac{K_{\nu+1}(x+i\varepsilon)}{K_\nu(x+i\varepsilon)}dx\\
&=\frac1{2\pi i}\int_0^{R\cos\theta_{R,\varepsilon}}
\frac1{(-x+i\varepsilon)^\alpha(-x+i\varepsilon-w)}
\frac{K_{\nu+1}(-x+i\varepsilon)}{K_\nu(-x+i\varepsilon)}dx.
\end{split}
\]
Note that
\[
K_\mu(e^{i\pi}z)=e^{-i\pi\mu}K_\mu(z)-i\pi I_\mu(z)
\]
for $z\in D$ and $\mu\geqq0$ (cf. \cite[p.80]{19}). Then, setting
\[
g_{\nu,1}(z)=\frac{K_{\nu+1}(z)+i\pi e^{i\pi\nu}I_{\nu+1}(z)}
{K_\nu(z)-i\pi e^{i\pi\nu}I_\nu(z)},
\]
we have
\begin{equation}
\frac{K_{\nu+1}(e^{i\pi}z)}{K_\nu(e^{i\pi}z)}=-g_{\nu,1}(z)
\label{2.15}
\end{equation}
and
\begin{equation}
\begin{split}
\varPsi_1(R,\alpha,\varepsilon)
&=\frac1{2\pi i}\int_0^{R\cos\theta_{R,\varepsilon}}
\frac{g_{\nu,1}(x-i\varepsilon)}
{e^{i\pi\alpha}(x-i\varepsilon)^\alpha (x-i\varepsilon+w)}dx\\
&=\frac1{2\pi i}\int_{\gamma_1^0}
\frac{g_{\nu,1}(z)}{e^{i\pi\alpha}z^\alpha(z+w)}dz,
\end{split}
\label{2.16}
\end{equation}
where $\gamma_1^0$ is the line in $D$ defined by
$\gamma_1^0\,:\,z=x-i\varepsilon,\,\,
0\leqq x\leqq R\cos\theta_{R,\varepsilon}$.

We now define three paths as follows:
\[
\begin{split}
&\gamma_1^1\,:\,z=\varepsilon e^{i\theta},
\,\,-\pi/2\leqq\theta\leqq0,\\
&\gamma_1^2\,:\,z=x,\,\,\varepsilon\leqq x\leqq R,\\
&\gamma_1^3\,:\,z=Re^{i\theta},
\,\,-\theta_{R,\varepsilon}\leqq\theta\leqq0.
\end{split}
\]
Since $w$ is inside $\gamma$, we have that $|\IM(w)|>\varepsilon$
if $\RE(w)<0$. There is no zero of $K_\nu$
on the real axis and the integrand of the right hand side of
\eqref{2.16} is
holomorphic inside and on the contour consisting of $\gamma_1^0$,
$\gamma_1^1$, $\gamma_1^2$ and $\gamma_1^3$.
Then it follows from the Cauchy integral theorem
\[
\varPsi_1(R,\alpha,\varepsilon)=\varPsi_1^1(\alpha,\varepsilon)
+\varPsi_1^2(R,\alpha,\varepsilon)-\varPsi_1^3(R,\alpha,\varepsilon),
\]
where
\[
\begin{split}
&\varPsi_1^1(\alpha,\varepsilon)=\frac1{2\pi}\int_{-\pi/2}^0
\frac{\varepsilon e^{i\theta}g_{\nu,1}(\varepsilon e^{i\theta})}
{e^{i(\theta+\pi)\alpha}\varepsilon^\alpha
(\varepsilon e^{i\theta}+w)}d\theta,\\
&\varPsi_1^2(R,\alpha,\varepsilon)=\frac1{2\pi i}\int_\varepsilon^R
\frac{g_{\nu,1}(x)}{e^{i\pi\alpha}x^\alpha(x+w)}dx,\\
&\varPsi_1^3(R,\alpha,\varepsilon)=\frac1{2\pi}
\int_{-\theta_{R,\varepsilon}}^0
\frac{R e^{i\theta}g_{\nu,1}(R e^{i\theta})}
{e^{i(\theta+\pi)\alpha}R^\alpha(R e^{i\theta}+w)}d\theta.
\end{split}
\]
If $-\pi/2\leqq\theta\leqq0$, we deduce from \eqref{2.9}
and \eqref{2.15}
\begin{equation}
|g_{\nu,1}(R e^{i\theta})|=\biggl|
\frac{K_{\nu+1}(R e^{i(\theta+\pi)})}{K_\nu(R e^{i(\theta+\pi)})}\biggr|
\leqq\frac{1+A_{\nu+1}/R}{1-A_\nu/R}
\label{2.17}
\end{equation}
for sufficiently large $R$ and
\[
|\varPsi_1^3(R,\alpha,\varepsilon)|\leqq\frac{\theta_{R,\varepsilon}}{2\pi}
\frac{R^{1-\alpha}}{R-|w|}
\frac{1+A_{\nu+1}/R}{1-A_\nu/R},
\]
which immediately implies
\[
\lim_{R\to\infty}\varPsi_1^3(R,\alpha,\varepsilon)=0.
\]
Moreover, from \eqref{2.17} for $R=x$ and $\theta=0$, it follows that
\[
\lim_{R\to\infty}\varPsi_1^2(R,\alpha,\varepsilon)
=\frac1{2\pi i}\int_\varepsilon^\infty
\frac{g_{\nu,1}(x)}{e^{i\pi\alpha}x^\alpha(x+w)}dx,
\]
for which we write $\varPsi_1^2(\alpha,\varepsilon)$.

In order to consider $\varPsi_2(R,\alpha,\varepsilon)$, we recall
\[
K_\mu(e^{-i \pi}z)=e^{i\pi\mu}K_\mu(z)+i\pi I_\mu(z)
\]
for $z\in D$ and $\mu\geqq0$ (cf. \cite[p.80]{19}).
Then we have
\[
\frac{K_{\nu+1}(e^{-i\pi}z)}{K_\nu(e^{-i\pi}z)}=-g_{\nu,2}(z),
\]
where
\[
g_{\nu,2}(z)=\frac{K_{\nu+1}(z)-i\pi e^{-i\pi\nu}I_{\nu+1}(z)}
{K_\nu(z)+i\pi e^{-i\pi\nu}I_\nu(z)}.
\]
In the same way as $\varPsi_1(R,\alpha,\varepsilon)$, we can show that
\[
\varPsi_2(R,\alpha,\varepsilon)=-\varPsi_2^1(\alpha,\varepsilon)
+\varPsi_2^2(R,\alpha,\varepsilon)+\varPsi_2^3(R,\alpha,\varepsilon),
\]
where
\[
\begin{split}
&\varPsi_2^1(\alpha,\varepsilon)=
\frac1{2\pi}\int_0^{\pi/2}
\frac{\varepsilon e^{i\theta}g_{\nu,2}(\varepsilon e^{i\theta})}
{e^{i(\theta-\pi)\alpha}\varepsilon^\alpha(\varepsilon e^{i\theta}+w)}
d\theta,\\
&\varPsi_2^2(R,\alpha,\varepsilon)=\frac1{2\pi i}\int_\varepsilon^R
\frac{g_{\nu,2}(x)}{e^{-i\pi\alpha}x^\alpha(x+w)}dx,\\
&\varPsi_2^3(R,\alpha,\varepsilon)=\frac1{2\pi}
\int_0^{\theta_{R,\varepsilon}} \frac{R e^{i\theta}g_{\nu,2}(R e^{i\theta})}
{e^{i(\theta-\pi)\alpha}R^\alpha(R e^{i\theta}+w)}d\theta.
\end{split}
\]
Similarly to $\varPsi_1^2(R,\alpha,\varepsilon)$ and
$\varPsi_1^3(R,\alpha,\varepsilon)$, it is easy to see that
\[
\lim_{R\to\infty}\varPsi_2^2(R,\alpha,\varepsilon)=\varPsi_2^2(\alpha,\varepsilon),\quad
\lim_{R\to\infty}\varPsi_2^3(R,\alpha,\varepsilon)=0,
\]
where
\[
\varPsi_2^2(\alpha,\varepsilon)=
\frac1{2\pi i}\int_\varepsilon^\infty
\frac{g_{\nu,2}(x)}{e^{-i\pi\alpha}x^\alpha(x+w)}dx.
\]
Therefore we conclude that
\begin{equation}
\lim_{R\to\infty}\varPsi(R,\alpha,\varepsilon)=
\varPsi_1^1(\alpha,\varepsilon)+\varPsi_1^2(\alpha,\varepsilon)
+\varPsi_2^1(\alpha,\varepsilon)-\varPsi_2^2(\alpha,\varepsilon)-\varPsi_3(\alpha,\varepsilon).
\label{2.18}
\end{equation}

We next consider the limiting behavior of the integral on \eqref{2.18}
as $\alpha,\varepsilon\to0$. We first let $\alpha\to0$.
Then it is easy to see
\[
\begin{split}
&\lim_{\alpha\to0}\varPsi_1^1(\alpha,\varepsilon)
=\frac1{2\pi}\int_{-\pi/2}^0
\frac{\varepsilon e^{i\theta}g_{\nu,1}(\varepsilon e^{i\theta})}
{\varepsilon e^{i\theta}+w}d\theta,\\
&\lim_{\alpha\to0}\varPsi_2^1(\alpha,\varepsilon)
=\frac1{2\pi}\int_0^{\pi/2}
\frac{\varepsilon e^{i\theta}g_{\nu,2}(\varepsilon e^{i\theta})}
{\varepsilon e^{i\theta}+w}d\theta,\\
&\lim_{\alpha\to0}\varPsi_3(\alpha,\varepsilon)
=\frac1{2\pi}\int_{-\pi/2}^{\pi/2}
\frac{\varepsilon e^{i\theta}}{\varepsilon e^{i\theta}-w}
\frac{K_{\nu+1}(\varepsilon e^{i\theta})}{K_\nu(\varepsilon e^{i\theta})}
d\theta.
\end{split}
\]
It is known that
\begin{equation}
K_\mu(z)=
\begin{cases}
\biggl(\log\dfrac2z\biggr)\{1+o(1)\}&\text{if $\mu=0$,}\\
\dfrac{\varGamma(\mu)}2\biggl(\dfrac2z\biggr)^\mu\{1+o(1)\}
\quad&\text{if $\mu>0$}
\end{cases}
\label{2.19}
\end{equation}
as $|z|\to0$ in $D$ (cf. \cite[p.512]{19}). Moreover, it is easy to see
\begin{equation}
I_\mu(z)=
\begin{cases}
1+o(1)&\text{if $\mu=0$,}\\
\dfrac1{\varGamma(\mu+1)}\biggl(\dfrac z2\biggr)^\mu\{1+o(1)\}
\quad&\text{if $\mu>0$}
\end{cases}
\label{2.20}
\end{equation}
as $|z|\to0$ in $D$ by the series expression. With the help of
\eqref{2.19} and \eqref{2.20}, we obtain that $z g_{\nu,1}(z)/(z+w)$ and
$z g_{\nu,2}(z)/(z+w)$ tend to $0$ if $\nu=0$ and
to $2\nu/w$ if $\nu>0$ as $|z|\to0$ in $D$.
Hence we have
\begin{equation}
\lim_{\varepsilon\to0}\lim_{\alpha\to0}\varPsi_1^1(\alpha,\varepsilon)
=\lim_{\varepsilon\to0}\lim_{\alpha\to0}\varPsi_2^1(\alpha,\varepsilon)
=\frac\nu{2w}.
\label{2.21}
\end{equation}
Moreover, since $z K_{\nu+1}(z)/(z-w)K_\nu(z)$ converges to
$-2\nu/w$ as $|z|\to0$ in $D$, we obtain
\[
\lim_{\varepsilon\to0}\lim_{\alpha\to0}\varPsi_3(\alpha,\varepsilon)
=-\frac\nu w.
\]
and
\begin{equation}
\lim_{\varepsilon\to0}\lim_{\alpha\to0}\{\varPsi_1^1(\alpha,\varepsilon)
+\varPsi_1^2(\alpha,\varepsilon)-\varPsi_3(\alpha,\varepsilon)\}
=\frac{2\nu}w.
\label{2.22}
\end{equation}

Net we set
\[
\varPhi(\alpha,\varepsilon)=\varPsi_1^2(\alpha,\varepsilon)
-\varPsi_2^2(\alpha,\varepsilon).
\]
Then, putting
\[
\begin{split}
F_\nu(x)=&-2i\sin(\pi\alpha)
\{K_{\nu+1}(x)K_\nu(x)-\pi^2 I_{\nu+1}(x)I_\nu(x)\}\\
&+i\pi e^{-i \pi\alpha}
\{e^{i\pi\nu}I_{\nu+1}(x)K_\nu(x)+e^{-i\pi\nu}K_{\nu+1}(x)I_\nu(x)\}\\
&+i\pi e^{i \pi\alpha}
\{e^{-i\pi\nu}I_{\nu+1}(x)K_\nu(x)+e^{i\pi\nu}K_{\nu+1}(x)I_\nu(x)\}.\\
\end{split}
\]
and
\[
h_\nu(x)=\frac{F_\nu(x)}{G_\nu(x)}=\frac{g_{\nu,1}(x)}{e^{i\pi\alpha}}
-\frac{g_{\nu,2}(x)}{e^{-i\pi\alpha}},
\]
where $G_\nu(x)$ is given by \eqref{2.2}, we have
\[
\varPhi(\alpha,\varepsilon)=\frac1{2\pi i}\int_\varepsilon^\infty
\frac{h_\nu(x)}{x^\alpha(x+w)}dx.
\]
Recall the formula
\[
K_{\nu+1}(x)I_\nu(x)+I_{\nu+1}(x)K_\nu(x)=\frac1x
\]
(cf. \cite[p.80]{19}). Then we get
\[
\begin{split}
&e^{-i \pi\alpha}
\{e^{i\pi\nu}I_{\nu+1}(x)K_\nu(x)+e^{-i\pi\nu}K_{\nu+1}(x)I_\nu(x)\}\\
&\hphantom{e^{-i \pi\alpha}}
+e^{i \pi\alpha}
\{e^{-i\pi\nu}I_{\nu+1}(x)K_\nu(x)+e^{i\pi\nu}K_{\nu+1}(x)I_\nu(x)\}\\
&=e^{-i \pi\alpha}\biggl\{\frac{e^{-i\pi\nu}}x
+2i\sin(\pi\nu)I_{\nu+1}(x)K_\nu(x)\biggr\}\\
&\hphantom{=e^{-i \pi\alpha}}
+e^{i \pi\alpha}\biggl\{\frac{e^{i\pi\nu}}x
-2i\sin(\pi\nu)I_{\nu+1}(x)K_\nu(x)\biggr\}\\
&=\frac{2\cos\pi(\alpha+\nu)}x
+4\sin(\pi\nu)\sin(\pi\alpha)I_{\nu+1}(x)K_\nu(x).
\end{split}
\]
Hence, letting
\[
\begin{split}
&F_{\nu,1}(x)=2i\sin(\pi\alpha)K_{\nu+1}(x)K_\nu(x),\quad
F_{\nu.2}(x)=2i\pi^2\sin(\pi\alpha)I_{\nu+1}(x)I_\nu(x),\\
&F_{\nu,3}(x)=\frac{2i\pi \cos\pi(\alpha+\nu)}x,\quad
F_{\nu,4}(x)=4i\pi\sin(\pi\nu)\sin(\pi\alpha)I_{\nu+1}(x)K_\nu(x),
\end{split}
\]
we have
\[
h_\nu(x)=\frac{-F_{\nu,1}(x)+F_{\nu,2}(x)
+F_{\nu,3}(x)+F_{\nu,4}(x)}{G_\nu(x)}.
\]
We set
\[
\varPhi_k(\alpha,\varepsilon)=\frac1{2\pi i}\int_\varepsilon^\infty
\frac1{x^\alpha(x+w)}\frac{F_{\nu,k}(x)}{G_\nu(x)}dx,
\quad 1\leqq k\leqq4.
\]

By virtue of Lemma \ref{Lemma 2.3}, the functions
\[
\frac1{x+w}
\frac{K_{\nu+1}(x)K_\nu(x)}{G_\nu(x)},\quad
\frac1{x+w}
\frac{I_{\nu+1}(x)K_\nu(x)}{G_\nu(x)}
\]
are integrable on $(\varepsilon,\infty)$. Hence we get
\begin{equation}
\lim_{\alpha\to0}\varPhi_1(\alpha,\varepsilon)=
\lim_{\alpha\to0}\varPhi_4(\alpha,\varepsilon)=0.
\label{2.23}
\end{equation}
The integral $\varPhi_3(\alpha,\varepsilon)$ is written as
\[
\varPhi_3(\alpha,\varepsilon)=\cos\pi(\alpha+\nu)\int_\varepsilon^\infty
\frac{dx}{x^{1+\alpha}(x+w)G_\nu(x)}
\]
and, with the help of Lemma \ref{Lemma 2.3}, we can derive
\begin{equation}
\lim_{\alpha\to0}\varPhi_3(\alpha,\varepsilon)=
\cos(\pi\nu)\int_\varepsilon^\infty \frac{dx}{x(x+w)G_\nu(x)}.
\label{2.24}
\end{equation}
It follows from \eqref{2.19} and \eqref{2.20} that
\begin{equation}
G_\nu(x)=
\begin{cases}
\biggl(\log\dfrac1x\biggr)^2\{1+o(1)\}&\text{if $\nu=0$,}\\
\dfrac1{\kappa_\nu x^{2\nu}}\{1+o(1)\}\quad&\text{if $\nu>0$}
\end{cases}
\label{2.25}
\end{equation}
as $x\to0$, where $\kappa_\nu=1/4^{\nu-1}\{\varGamma(\nu)\}^2$.
This implies the convergence of the right hand side of \eqref{2.24}
as $\varepsilon\to0$ and 
\begin{equation}
\lim_{\varepsilon\to0}\lim_{\alpha\to0}\varPhi_3(\alpha,\varepsilon)=
\cos(\pi\nu)\int_0^\infty \frac{dx}{x(x+w)G_\nu(x)}.
\label{2.26}
\end{equation}

It remains to consider $\varPhi_2(\alpha,\varepsilon)$,
\[
\varPhi_2(\alpha,\varepsilon)=
\pi\sin(\pi\alpha)\int_\varepsilon^\infty\frac1{x^\alpha(x+w)}
\frac{I_{\nu+1}(x)I_\nu(x)}{G_\nu(x)}dx.
\]
We should remark that the function
\[
\frac1{x+w}\frac{I_{\nu+1}(x)I_\nu(x)}{G_\nu(x)}
\]
is not integrable on $(\varepsilon, \infty)$. We write
\[
\varPhi_2(\alpha,\varepsilon)=-\varPhi_2^1(\alpha,\varepsilon)
-\varPhi_2^2(\alpha,\varepsilon)
+\varPhi_2^3(\alpha,\varepsilon)+\varPhi_2^4(\alpha,\varepsilon),
\]
where
\[
\begin{split}
&\varPhi_2^1(\alpha,\varepsilon)=\frac{\sin(\pi\alpha)}\pi
\int_\varepsilon^\infty\frac1{x^\alpha(x+w)}
\frac{K_\nu(x)^2}{G_\nu(x)}dx,\\
&\varPhi_2^2(\alpha,\varepsilon)=2\sin(\pi\nu)\sin(\pi\alpha)
\int_\varepsilon^\infty\frac1{x^\alpha(x+w)}
\frac{I_{\nu+1}(x)K_\nu(x)}{G_\nu(x)}dx,\\
&\varPhi_2^3(\alpha,\varepsilon)=\pi\sin(\pi\alpha)
\int_\varepsilon^\infty\frac1{x^\alpha(x+w)}
\frac{I_\nu(x)\{I_\nu(x)-I_{\nu+1}(x)\}}{G_\nu(x)}dx,\\
&\varPhi_2^4(\alpha,\varepsilon)=\frac{\sin(\pi\alpha)}\pi
\int_\varepsilon^\infty\frac{dx}{x^\alpha(x+w)}.\\
\end{split}
\]
We can easily deduce from Lemma \ref{Lemma 2.3}
\[
\lim_{\alpha\to0}\varPhi_2^1(\alpha,\varepsilon)=
\lim_{\alpha\to0}\varPhi_2^2(\alpha,\varepsilon)=0
\]
in the same way as \eqref{2.23}.

To compute $\varPhi_2^3(\alpha,\varepsilon)$,
we note the following:
\[
\frac{I_\nu(x)\{I_\nu(x)-I_{\nu+1}(x)\}}{G_\nu(x)}
=\frac{2\nu+1}{\pi^2x}\{1+o(1)\},\,
x\to\infty,
\]
which is obtained from \eqref{2.10}, \eqref{2.11}
and Lemma \ref{Lemma 2.3}. Since
\[
\frac1{x+w}\frac{I_\nu(x)\{I_\nu(x)-I_{\nu+1}(x)\}}{G_\nu(x)}
\]
is integrable on $(\varepsilon,\infty)$, we have that
$\varPhi_2^3(\alpha,\varepsilon)$ tends to $0$ as $\alpha\to0$.

The calculation of $\varPhi_2^4(\alpha,\varepsilon)$ is easy.
In fact we have
\[
\begin{split}
\varPhi_2^4(\alpha,\varepsilon)
&=\frac{\sin(\pi\alpha)}\pi
\int_\varepsilon^\infty\frac{dx}{x^{1+\alpha}}
-\frac{w\sin(\pi\alpha)}\pi
\int_\varepsilon^\infty\frac{dx}{x^{1+\alpha}(x+w)}\\
&=\frac{\sin(\pi\alpha)}{\pi\alpha\varepsilon^\alpha}
-\frac{w\sin(\pi\alpha)}\pi
\int_\varepsilon^\infty\frac{dx}{x^{1+\alpha}(x+w)}\to1,\, \alpha\to0
\end{split}
\]
for any $\varepsilon>0$. Hence we conclude from \eqref{2.23} and \eqref{2.26}
\begin{equation}
\lim_{\varepsilon\to0}\lim_{\alpha\to0}\varPhi(\alpha,\varepsilon)=1+
\cos(\pi\nu)\int_0^\infty \frac{dx}{x(x+w)G_\nu(x)}.
\label{2.27}
\end{equation}

Combining \eqref{2.14}, \eqref{2.18}, \eqref{2.21},
\eqref{2.22} and \eqref{2.27}, we complete our proof of
\eqref{2.5} and \eqref{2.6} in the case of $\nu\geqq0$.

\medskip

It may be worthwhile to mention that the method used
to prove Theorem \ref{Theorem 2.1} can be applied to the decomposition
of $K_{\nu+\rho}/K_\nu$, and we can show the following.

\medskip

\begin{thm}\label{Theorem 2.5}
Let $\nu$ and $\rho$ be real numbers
with $\nu\geqq0$, $-\nu\leqq \rho<1$, $\rho\neq0$.
In the case when there is no integer $n$ such that
$\nu=2n+3/2$, we have that, if $\nu<3/2$,
\[
\frac{K_{\nu+\rho}(w)}{K_\nu(w)}=1+\int_0^\infty
\frac1{x+w}\frac{H_{\nu,\rho}(x)}{G_\nu(x)}dx
\]
and that, if $\nu>3/2$,
\[
\frac{K_{\nu+\rho}(w)}{K_\nu(w)}=1
+\sum_{j=1}^{N(\nu)}\frac1{z_{\nu,j}-w}
\frac{K_{\nu+\rho}(z_{\nu,j})}{K_{\nu+1}(z_{\nu,j})}
+\int_0^\infty\frac1{x+w}\frac{H_{\nu,\rho}(x)}{G_\nu(x)}dx,
\]
where
\[
\begin{split}
H_{\nu,\rho}(x)=&
-\cos\pi(\nu+\rho)K_{\nu+\rho}(x)I_\nu(x)\\
&+\cos(\pi\nu)I_{\nu+\rho}(x)K_\nu(x)
+\frac{\sin(\pi\rho)}\pi K_{\nu+\rho}(x)K_\nu(x).
\end{split}
\]
When $\nu=2n+3/2$ for some integer $n$, we have that
\[
\frac{K_{\nu+\rho}(w)}{K_\nu(w)}=1+\varLambda_{\nu,\rho}
+\sum_{j=1}^{N(\nu)}\frac{\varLambda_{\nu,\rho}(z_{\nu,j})}{z_{\nu,j}-w},
\]
where
\[
\begin{split}
&\varLambda_{\nu,\rho}=\lim_{\delta\to0}\frac{\sin(\pi\rho)}\pi
\int_{|x-x_0|>\delta}\frac1{x+w}
\frac{K_{\nu+\rho}(x)}{K_\nu(x)-\pi I_\nu(x)}dx,\\
&\varLambda_{\nu,\rho}(z)=
\begin{cases}
\dfrac{K_{\nu+\rho}(z)}{K_{\nu+1}(z)}\quad&\text{if $z\in D_{\nu+1}$,}\\
\dfrac{-\cos(\pi\rho)K_{\nu+\rho}(-z)+\pi I_{\nu+\rho}(-z)}
{K_{\nu+1}(-z)+\pi I_{\nu+1}(-z)}\quad&\text{if $z\in(-\infty,0)$}
\end{cases}
\end{split}
\]
and $x_0$ is the unique real zero of $K_\nu$.
\end{thm}

\medskip

With the help of $K_\rho=K_{-\rho}$ and \eqref{2.8},
we furthermore see that Theorem \ref{Theorem 2.5} gives
a representation for $K_\mu/K_\nu$
for every $\mu,\nu\in\mathbb R$.


\section{The first hitting time of the Bessel process}

\noindent
For $\nu\in\mathbb R$ the one-dimensional diffusion process
with infinitesimal generator
\[
\mathcal G^{(\nu)}=
\frac12\frac{d^2}{dx^2}+\frac{2\nu+1}{2x}\frac d{dx}
=\frac1{2x^{2\nu+1}}\frac d{dx}\biggl(x^{2\nu+1}\frac d{dx}\biggr)
\]
is called the Bessel process with index $\nu$.
The classification of boundary points gives the following information.
The endpoint $\infty$ is a natural boundary for any $\nu\in\mathbb R$.
For $\nu\geqq0$, $0$ is an entrance and not exit boundary.
For $-1<\nu<0$, $0$ is a regular boundary, which is instantly
reflecting. For $\nu\leqq-1$, $0$ is an exit but not entrance boundary.
For more details, see \cite{11} and \cite{17} for example. 

For $a,b\in\mathbb R$ let $\tau_{a,b}^{(\nu)}$ be the first hitting time
to $b$ of the Bessel process with index $\nu$ starting at $a$.
The conditional distribution of $\tau_{a,b}^{(\nu)}$
under $\tau_{a,b}^{(\nu)}<\infty$ is infinitely divisible.
The purpose of this section is to give the exact form of
the L{\'e}vy measure $m_{a,b}^{(\nu)}$ when $0\leqq b<a$
by applying Theorem \ref{Theorem 2.1}.

It is known that, when $0\leqq a<b$, the distribution of
$\tau_{a,b}^{(\nu)}$ is a mixture of exponential distributions.
Let us recall the results in \cite{2}. See also \cite{11}.
In this case, the Laplace transforms of
the conditional distributions are given by the following.
For $\lambda>0$, if $b>0$ and $\nu>-1$,
\[
E[e^{-\lambda\tau_{0,b}^{(\nu)}}\,|\,\tau_{0,b}^{(\nu)}<\infty]=
\frac{(b\sqrt{2\lambda})^\nu}{2^\nu\varGamma(\nu+1)}
\frac1{I_\nu(b\sqrt{2\lambda})},
\]
if $0<a\leqq b$ and $\nu>-1$,
\[
E[e^{-\lambda\tau_{a,b}^{(\nu)}}\,|\,\tau_{a,b}^{(\nu)}<\infty]=
\biggl(\frac ba\biggr)^\nu
\frac{I_\nu(a\sqrt{2\lambda})}{I_\nu(b\sqrt{2\lambda})},
\]
if $0<a\leqq b$ and $\nu\leqq-1$,
\[
E[e^{-\lambda\tau_{a,b}^{(\nu)}}\,|\,\tau_{a,b}^{(\nu)}<\infty]=
\biggl(\frac ab\biggr)^\nu
\frac{I_{-\nu}(a\sqrt{2\lambda})}{I_{-\nu}(b\sqrt{2\lambda})}.
\]
Combining these results with the formula
\[
I_\mu(x)=\biggl(\frac x2\biggr)^\mu \frac1{\varGamma(\mu+1)}
\prod_{n=1}^\infty \biggl(1+\frac{x^2}{j_{\mu,n}^2}\biggr)
\]
for $\mu>-1$ and $x>0$, where $\{j_{\mu,n}\}_{n=1}^\infty$ is
an increasing sequence of positive zeros of the Bessel function
$J_\mu$ of the first kind of order $\mu$, we obtain the following
expressions for the L\'evy measures: if $b>0$ and $\nu>-1$,
\[
\frac{d m_{0,b}^{(\nu)}(x)}{dx}=\frac{1_{(0,\infty)}(x)}x
\sum_{n=1}^\infty
e^{-\frac{j_{\nu,n}^2}{2b^2}x},
\]
if $0<a\leqq b$ and $\nu>-1$,
\[
\frac{d m_{a,b}^{(\nu)}(x)}{dx}=\frac{1_{(0,\infty)}(x)}x
\sum_{n=1}^\infty
\biggl( e^{-\frac{j_{\nu,n}^2}{2b^2}x}
-e^{-\frac{j_{\nu,n}^2}{2a^2}x}\biggr),
\]
if $0<a\leqq b$ and $\nu\leqq-1$,
\[
\frac{d m_{a,b}^{(\nu)}(x)}{dx}=\frac{1_{(0,\infty)}(x)}x
\sum_{n=1}^\infty
\biggl( e^{-\frac{j_{-\nu,n}^2}{2b^2}x}
-e^{-\frac{j_{-\nu,n}^2}{2a^2}x}\biggr),
\]
respectively, where $1_A$ is the indicator function of a set $A$.

The following is the main result in this section.

\medskip

\begin{thm}\label{Theorem 3.1}
For $0\leqq b<a$ the support of the L\'evy measure
$m_{a,b}^{(\nu)}$ is $[0,\infty)$ and it is absolutely
continuous with respect to the Lebesgue measure.
We have the following expressions for the density
$p_{a,b}^{(\nu)}$, $x>0$ of $m_{a,b}^{(\nu)}$.

\noindent
{\rm (1)} If $a>0$,
\[
p_{a,0}^{(-1/2)}(x)=\frac a{\sqrt{2\pi x^3}}.
\]
{\rm (2)} If $a>0$, $\nu-1/2\in\mathbb Z$ and $\nu\leqq-3/2$,
\[
p_{a,0}^{(\nu)}(x)=\frac a{\sqrt{2\pi x^3}}
-\frac1{2\sqrt{\pi x^3}}\sum_{j=1}^{N(\nu)}\int_0^\infty
e^{-\frac{\xi^2}{4x}+\frac{z_{\nu,j}\xi}{\sqrt 2 a}}d\xi.
\]
{\rm (3)} If $a>0$, $-3/2<\nu<0$ and $\nu\neq-1/2$,
\[
p_{a,0}^{(\nu)}(x)=\frac a{\sqrt{2\pi x^3}}
+\frac{\cos(\pi\nu)}{2\sqrt{\pi x^3}}
\int_0^\infty \int_0^\infty \frac1{\eta G_{|\nu|}(\eta)}
e^{-\frac{\xi^2}{4x}+\frac{\xi\eta}{\sqrt 2 a}}d\xi d\eta.
\]
{\rm (4)} If $a>0$, $\nu-1/2\notin\mathbb Z$ and $\nu<-3/2$,
\[
\begin{split}
p_{a,0}^{(\nu)}(x)=\frac a{\sqrt{2\pi x^3}}
&-\frac1{2\sqrt{\pi x^3}}\sum_{j=1}^{N(\nu)}\int_0^\infty
e^{-\frac{\xi^2}{4x}+\frac{z_{\nu,j}\xi}{\sqrt 2 a}}d\xi\\
&+\frac{\cos(\pi\nu)}{2\sqrt{\pi x^3}}
\int_0^\infty\!\!\!\int_0^\infty \frac1{\eta G_{|\nu|}(\eta)}
e^{-\frac{\xi^2}{4x}+\frac{\xi\eta}{\sqrt 2 a}}d\xi d\eta.
\end{split}
\]
{\rm (5)} If $0<b<a$ and $\nu=\pm1/2$,
\[
p_{a,b}^{(\nu)}(x)=\frac{a-b}{\sqrt{2\pi x^3}}.
\]
{\rm (6)} If $0<b<a$, $\nu-1/2\in\mathbb Z$ and $|\nu|\geqq3/2$
\[
p_{a,b}^{(\nu)}(x)=\frac{a-b}{\sqrt{2\pi x^3}}
-\frac1{2\sqrt{\pi x^3}}\sum_{j=1}^{N(\nu)}\int_0^\infty
e^{-\frac{\xi^2}{4x}} \biggl(e^{\frac{z_{\nu,j}\xi}{\sqrt 2 a}}
-e^{\frac{z_{\nu,j}\xi}{\sqrt 2 b}}\biggr)d\xi.
\]
{\rm (7)} If $0<b<a$, $0\leqq|\nu|<3/2$ and $\nu\neq\pm1/2$,
\[
p_{a,b}^{(\nu)}(x)=\frac{a-b}{\sqrt{2\pi x^3}}
+\frac{\cos(\pi\nu)}{2\sqrt{\pi x^3}}
\int_0^\infty \!\!\! \int_0^\infty \frac1{\eta G_{|\nu|}(\eta)}
e^{-\frac{\xi^2}{4x}}\biggl(e^{-\frac{\xi\eta}{\sqrt 2 a}}
-e^{-\frac{\xi\eta}{\sqrt 2 b}}\biggr)d\xi d\eta.
\]
{\rm (8)} If $0<b<a$, $\nu-1/2\notin\mathbb Z$ and $|\nu|>3/2$,
\[
\begin{split}
p_{a,b}^{(\nu)}(x)=
\frac{a-b}{\sqrt{2\pi x^3}}
&-\frac1{2\sqrt{\pi x^3}}\sum_{j=1}^{N(\nu)}\int_0^\infty
e^{-\frac{\xi^2}{4x}} \biggl(e^{\frac{z_{\nu,j}\xi}{\sqrt 2 a}}
-e^{\frac{z_{\nu,j}\xi}{\sqrt 2 b}}\biggr)d\xi\\
&+\frac{\cos(\pi\nu)}{2\sqrt{\pi x^3}}
\int_0^\infty \!\!\! \int_0^\infty \frac1{\eta G_{|\nu|}(\eta)}
e^{-\frac{\xi^2}{4x}}\biggl(e^{-\frac{\xi\eta}{\sqrt 2 a}}
-e^{-\frac{\xi\eta}{\sqrt 2 b}}\biggr)d\xi d\eta.
\end{split}
\]
\end{thm}

\medskip

The rest of this section is devoted to the proof of Theorem \ref{Theorem 3.1}.
By the formulae for Laplace transforms for $\tau_{a,b}^{(\nu)}$,
we have, for $\lambda>0$, if $a>0$ and $\nu<0$,
\begin{equation}
E[e^{-\lambda\tau_{a,0}^{(\nu)}}\,|\,\tau_{a,0}^{(\nu)}<\infty]=
\frac{2^{\nu+1}}{\varGamma(|\nu|)(a\sqrt{2\lambda})^\nu}
K_\nu(a\sqrt{2\lambda}),
\label{3.1}
\end{equation}
if $0<b\leqq a$ and $\nu\in\mathbb R$,
\begin{equation}
E[e^{-\lambda\tau_{a,b}^{(\nu)}}\,|\,\tau_{a,b}^{(\nu)}<\infty]=
\biggl(\frac ab\biggr)^{|\nu|}
\frac{K_{|\nu|}(a\sqrt{2\lambda})}{K_{|\nu|}(b\sqrt{2\lambda})}.
\label{3.2}
\end{equation}

We represent $K_\nu(x)$ by means of $G_\nu$ and zeros of $K_\nu$.

\medskip

\begin{prop}\label{Proposition 3.2}
For $x>0$ we have the following formulae.

\noindent
{\rm (1)} If $\mu=1/2$,
\begin{equation}
\log \{x^{1/2} K_{1/2}(x)\}=\frac12\log\frac\pi2-x.
\label{3.3}
\end{equation}
{\rm (2)} If $\mu-1/2\in\mathbb Z$ and $\mu\geqq3/2$,
\begin{equation}
\log \{x^\mu K_\mu(x)\}=\log\{2^{\mu-1}\varGamma(\mu)\}
-x-\sum_{j=1}^{N(\mu)}\log \frac{z_{\mu,j}}{z_{\mu,j}-x}.
\label{3.4}
\end{equation}
{\rm (3)} If $0<\mu<3/2$ and $\mu\neq1/2$,
\begin{equation}
\log \{x^\mu K_\mu(x)\}=
\log\{2^{\mu-1}\varGamma(\mu)\}-x
-\cos(\pi\mu)\int_0^\infty \frac1{y G_\mu(y)}\log\frac{y+x}y dy.
\label{3.5}
\end{equation}
{\rm (4)} If $\mu-1/2\notin\mathbb Z$ and $\mu>3/2$,
\begin{equation}
\begin{split}
\log \{x^\mu K_\mu(x)\}=&\log\{2^{\mu-1}\varGamma(\mu)\}
-x-\sum_{j=1}^{N(\mu)}\log \frac{z_{\mu,j}}{z_{\mu,j}-x}\\
&-\cos(\pi\mu)\int_0^\infty \frac1{y G_\mu(y)}\log\frac{y+x}y dy.
\end{split}
\label{3.6}
\end{equation}
\end{prop}
\begin{proof}  Formula \eqref{3.3} is obtained from
the explicit expression for $K_{1/2}(x)$.

In order to show the others, we note the following formula:
\begin{equation}
\frac d{dx}\log\{x^\nu K_\nu(x)\}
=-\frac{K_{\nu+1}(x)}{K_\nu(x)}+\frac{2\nu}x,
\label{3.7}
\end{equation}
which can be derived by \eqref{2.7} and
$(x^\nu K_\nu(x))'=-x^\nu K_{\nu-1}(x)$
(cf. \cite[p.79]{19}).

If $\nu-1/2\in\mathbb Z$ and $\nu\geqq3/2$,
it follows from \eqref{2.4} and \eqref{3.7} that
\[
\frac d{dx}\log\{x^\nu K_\nu(x)\}
=-1-\sum_{j=1}^{N(\nu)}\frac1{z_{\nu,j}-x},
\]
Then we obtain that, for any $\varepsilon>0$,
\[
\log\{x^\nu K_\nu(x)\}-\log\{\varepsilon^\nu K_\nu(\varepsilon)\}
=-(x-\varepsilon)
+\sum_{j=1}^{N(\nu)}\int_\varepsilon^x \frac{d\xi}{\xi-z_{\nu,j}}.
\]
Hence, letting $\varepsilon\to0$, we get \eqref{3.4} with the help of
\begin{equation}
\lim_{x\to0} x^\nu K_\nu(x)=2^{\nu-1}\varGamma(\nu).
\label{3.8}
\end{equation}

If $0<\nu<3/2$ and $\nu\neq1/2$, it follows from
\eqref{2.5} and \eqref{3.7} that
\[
\frac d{dx}\log\{x^\nu K_\nu(x)\}
=-1-\cos(\pi\nu)\int_0^\infty\frac{dy}{y(y+x)G_\nu(y)}.
\]
Hence we have
\[
\log\{x^\nu K_\nu(x)\}-\log\{\varepsilon^\nu K_\nu(\varepsilon)\}
=-(x-\varepsilon)-\cos(\pi\nu)\int_\varepsilon^x d\xi
\int_0^\infty\frac{dy}{y(y+\xi) G_\nu(y)}.
\]
Note that $G_\nu(x)$ is positive for $x>0$
unless $\nu$ is $2n+3/2$ for any integer $n$.
Thus it follows from \eqref{3.8} that
\[
\log\{x^\nu K_\nu(x)\}-\log\{2^{\nu-1}\varGamma(\nu)\}
=-x-\cos(\pi\nu)\int_0^x d\xi
\int_0^\infty\frac{dy}{y(y+\xi) G_\nu(y)}.
\]
By the Fubini theorem, the right hand side is equal to
\[
-x-\cos(\pi\nu)\int_0^\infty \frac{dy}{y G_\nu(y)}
\int_0^x\frac{d\xi}{y+\xi},
\]
which yields \eqref{3.5}.

We can show \eqref{3.6} in the same way.
\end{proof}

\medskip

In order to see Theorem \ref{Theorem 3.1}, we only have to check
\[
\phi_{a,b}^{(\nu)}(\lambda):=
\log E[e^{-\lambda\tau_{a,b}^{(\nu)}}\,|\,\tau_{a,b}^{(\nu)}<\infty]
=\int_0^\infty (e^{-\lambda x}-1)p_{a,b}^{(\nu)}(x)dx
\]
for each case. It follows from \eqref{3.1} and \eqref{3.2} that,
if $a>0$ and $\nu<0$,
\[
\phi_{a,0}^{(\nu)}(\lambda)
=\log\{(a\sqrt{2\lambda})^{|\nu|}K_{|\nu|}(a\sqrt{2\lambda})\}
-\log\{2^{|\nu|-1}\varGamma(|\nu|)\},
\]
and that, if $0<b<a$ and $\nu\in\mathbb R$,
\begin{equation}
\phi_{a,b}^{(\nu)}(\lambda)
=\log\{(a\sqrt{2\lambda})^{|\nu|}K_{|\nu|}(a\sqrt{2\lambda})\}
-\log\{(b\sqrt{2\lambda})^{|\nu|}K_{|\nu|}(b\sqrt{2\lambda})\}.
\label{3.9}
\end{equation}

The following lemma gives rise to Theorem \ref{Theorem 3.1} for $\nu\neq0$
by Proposition \ref{Proposition 3.2}.

\medskip

\begin{lemma}\label{Lemma 3.3}
Let $c>0$, $\nu>0$ and $z\in{\mathbb C}^-$.
For $\lambda>0$ it follows that
\begin{align}
&\sqrt{2\lambda}=-\int_0^\infty
\frac{e^{-\lambda x}-1}{\sqrt{2\pi x^3}}dx,
\label{3.10}\\
&\log\frac z{z-c\sqrt{2\lambda}}=\int_0^\infty
\frac{e^{-\lambda x}-1}{2\sqrt{\pi x^3}}dx
\int_0^\infty e^{-\frac{\xi^2}{4x}
+\frac{z\xi}{\sqrt2 c}}d\xi
\label{3.11}
\end{align}
and
\begin{equation}
\begin{split}
&\int_0^\infty \log\frac{\eta+c\sqrt{2\lambda}}\eta
\frac1{\eta G_\nu(\eta)}d\eta\\
&=-\int_0^\infty \frac{e^{-\lambda x}-1}{2\sqrt{\pi x^3}}dx
\int_0^\infty \!\!\! \int_0^\infty \frac1{\eta G_\nu(\eta)}
e^{-\frac{\xi^2}{4x}-\frac{\xi\eta}{\sqrt2 c}}d\xi d\eta.
\end{split}
\label{3.12}
\end{equation}
\end{lemma}
\begin{proof}
We recall the following formula (cf. \cite[p.361]{3}):
for $-1<p<0$,
\[
\int_0^\infty y^{p-1}(e^{-y}-1)dy=\varGamma(p).
\]
Setting $p=-1/2$ and noting $\varGamma(-1/2)=-2\sqrt\pi$,
we obtain \eqref{3.10}.

Recall the formulae
\begin{align}
&\int_0^\infty e^{-\lambda x-\frac{\xi^2}{4x}} x^{-3/2}dx
=\frac{2\sqrt\pi}\xi e^{-\xi\sqrt\lambda},
\label{3.13}\\
&\int_0^\infty e^{-\frac{\xi^2}{4x}} x^{-3/2}dx
=2\int_0^\infty e^{-\frac{\xi^2 y^2}4}dy
=\frac{2\sqrt\pi}\xi,\quad\lambda,\xi>0.
\label{3.14}
\end{align}
and, for $\alpha,\beta\in\mathbb C$ with $\RE(\alpha)>0$ and $\RE(\beta)>0$
(cf. \cite[p.361]{3}),
\begin{equation}
\int_0^\infty \frac{e^{-\alpha x}-e^{-\beta x}}xdx=
\log\frac\beta\alpha
\label{3.15}
\end{equation}
Then we obtain
\[
\begin{split}
\int_0^\infty \!\!\! \int_0^\infty \frac{1-e^{-\lambda x}}{x^{3/2}}
e^{-\frac{\xi^2}{4x}-\alpha \xi}dxd\xi
&=2\sqrt\pi \int_0^\infty
\frac{e^{-\alpha\xi}-e^{-(\sqrt\lambda+\alpha)\xi}}\xi d\xi\\
&=2\sqrt\pi \log\frac{\sqrt\lambda+\alpha}\alpha.
\end{split}
\]
Hence, applying the Fubini theorem
to the right hand side of \eqref{3.11}, we obtain
\[
\begin{split}
\frac1{2\sqrt\pi}\int_0^\infty e^{\frac{z\xi}{\sqrt2 c}} d\xi
\int_0^\infty \frac{e^{-\lambda x}-1}{x^{3/2}}
e^{-\frac{\xi^2}{4x}}dx
&=\int_0^\infty e^{\frac{z\xi}{\sqrt2 c}}
\frac{e^{-\xi\sqrt\lambda}-1}\xi d\xi\\
&=\log \frac{-z/\sqrt2 c}{\sqrt\lambda-z/\sqrt2 c}
\end{split}
\]

The calculation of the right hand side of \eqref{3.12} is similar.
Since $G_\nu(x)$ is positive for $x>0$, we may apply 
the Fubini theorem and, applying \eqref{3.13}, \eqref{3.14}
and \eqref{3.15} again, we get
\[
\begin{split}
&\int_0^\infty \frac{e^{-\lambda x}-1}{2\sqrt{\pi x^3}} dx
\int_0^\infty \!\!\! \int_0^\infty \frac1{\eta G_\nu(\eta)}
e^{-\frac{\xi^2}{4x}-\frac{\xi\eta}{\sqrt2 c}}d\xi d\eta\\
&=\frac1{2\sqrt\pi}\int_0^\infty\frac{d\eta}{\eta G_\nu(\eta)}
\int_0^\infty e^{-\frac{\xi\eta}{\sqrt 2 c}} d\xi
\int_0^\infty \frac{e^{-\lambda x}-1}{x^{3/2}}
e^{-\frac{\xi^2}{4x}}dx\\
&=\int_0^\infty\frac{d\eta}{\eta G_\nu(\eta)}
\int_0^\infty e^{-\frac{\xi\eta}{\sqrt 2 c}}
\frac{e^{-\xi\sqrt\lambda-1}}\xi d\xi\\
&=\int_0^\infty\frac1{\eta G_\nu(\eta)}
\log \frac{\eta/\sqrt 2 c}{\sqrt\lambda-\eta/\sqrt 2 c}d\eta.
\end{split}
\]
This immediately implies \eqref{3.12}.
\end{proof}

\medskip

\begin{remark}\label{Remark 3.4}
If $\nu=0$, it follows from \eqref{2.25}
that the left hand side of \eqref{3.12} diverges.
\end{remark}

\medskip

We finally consider the case of $\nu=0$. Since $K_0'(x)=-K_1(x)$, we have
\[
\frac d{dx}\log K_0(x)=-\frac{K_1(x)}{K_0(x)}
=-1-\int_0^\infty \frac{d\eta}{\eta(\eta+x)G_0(\eta)}
\]
by Theorem \ref{Theorem 2.1}. Hence, by \eqref{3.9}, we get
\begin{equation}
\begin{split}
\phi_{a,b}^{(0)}(\lambda)
&=-(a-b)\sqrt{2\lambda}
-\int_{b\sqrt{2\lambda}}^{a\sqrt{2\lambda}} d\xi
\int_0^\infty\frac{d\eta}{\eta(\eta+x)G_0(\eta)}\\
&=-(a-b)\sqrt{2\lambda}-\int_0^\infty\frac1{\eta G_0(\eta)}
\log\frac{\eta+a\sqrt{2\lambda}}{\eta+b\sqrt{2\lambda}}d\eta.
\end{split}
\label{3.16}
\end{equation}

Note that
\[
\frac1{\eta G_0(\eta)}
\log\frac{\eta+a\sqrt{2\lambda}}{\eta+b\sqrt{2\lambda}}
\]
is integrable on $(0,\infty)$ by \eqref{2.25}.
For $c>0$ and $\varepsilon>0$, we write
\[
\int_\varepsilon^\infty \log\frac{\eta+c\sqrt{2\lambda}}\eta
\frac1{\eta G_0(\eta)}d\eta
=-\int_0^\infty \frac{e^{-\lambda x}-1}{2\sqrt{\pi x^3}} dx
\int_\varepsilon^\infty \frac{d\eta}{\eta G_0(\eta)}
\int_0^\infty e^{-\frac{\xi^2}{4x}-\frac{\xi\eta}{\sqrt2 c}} d\xi.
\]
This formula immediately implies
\begin{equation}
\begin{split}
&\int_\varepsilon^\infty\frac1{\eta G_0(\eta)}
\log\frac{\eta+a\sqrt{2\lambda}}{\eta+b\sqrt{2\lambda}}d\eta\\
&=\int_0^\infty \frac{e^{-\lambda x}-1}{2\sqrt{\pi x^3}} dx
\int_\varepsilon^\infty \frac{d\eta}{\eta G_0(\eta)}
\int_0^\infty e^{-\frac{\xi^2}{4x}}
\biggl(e^{-\frac{\xi\eta}{\sqrt2 b}}-e^{-\frac{\xi\eta}{\sqrt2 a}}
\biggr) d\xi.
\end{split}
\label{3.17}
\end{equation}
If we show that the integrand is integrable on
$(0,\infty)\times(0,\infty)\times(0,\infty)$, we can conclude
\[
\begin{split}
&\lim_{\varepsilon\to0}\int_\varepsilon^\infty
\log\frac{\eta+c\sqrt{2\lambda}}\eta
\frac1{\eta G_0(\eta)}d\eta\\
&=\int_0^\infty \frac{e^{-\lambda x}-1}{2\sqrt{\pi x^3}} dx
\int_0^\infty \frac{d\eta}{\eta G_0(\eta)}
\int_0^\infty e^{-\frac{\xi^2}{4x}}
\biggl(e^{-\frac{\xi\eta}{\sqrt2 b}}-e^{-\frac{\xi\eta}{\sqrt2 a}}
\biggr) d\xi
\end{split}
\]
and obtain Theorem \ref{Theorem 3.1} in the case of $\nu=0$.

Since $\lambda>0$ and $0<b<a$, the integrand on \eqref{3.17} is non-negative.
We have that
\[
\begin{split}
0&\leqq \frac{e^{-\lambda x}-1}{2\sqrt{\pi x^3}}
\frac1{\eta G_0(\eta)}e^{-\frac{\xi^2}{4x}}
\biggl(e^{-\frac{\xi\eta}{\sqrt2 b}}-e^{-\frac{\xi\eta}{\sqrt2 a}}
\biggr)\\
&=\frac{1-e^{-\lambda x}}{2\sqrt{\pi x^3}}
\frac1{\eta G_0(\eta)}e^{-\frac{\xi^2}{4x}}e^{-\frac{\xi\eta}{\sqrt2 a}}
\biggl(1-e^{-\frac{\xi\eta}{\sqrt2 b}+\frac{\xi\eta}{\sqrt2 a}}
\biggr)\\
&\leqq\frac{1-e^{-\lambda x}}{2\sqrt{2\pi x^3}}
\frac\xi{ G_0(\eta)}e^{-\frac{\xi^2}{4x}}e^{-\frac{\xi\eta}{\sqrt2 a}}
\biggl(\frac1b-\frac1a\biggr).
\end{split}
\]
From \eqref{3.12} and \eqref{3.13} we obtain
\[
\begin{split}
\int_0^\infty \!\!\! \int_0^\infty \xi e^{-\frac{\xi^2}{4x}}
e^{-\frac{\xi\eta}{\sqrt2 a}} \frac{1-e^{-\lambda x}}{x^{3/2}}d\xi dx
&=2\sqrt\pi\int_0^\infty e^{-\frac{\xi\eta}{\sqrt2 a}}
(1-e^{-\xi\sqrt\lambda}) d\xi\\
&\leqq 2\sqrt\pi\int_0^\infty e^{-\frac{\xi\eta}{\sqrt2 a}} d\xi
=\frac{2\sqrt{2\pi}a}\eta.
\end{split}
\]
Hence the integral
\[
\int_0^\infty \!\!\! \int_0^\infty \frac{e^{-\lambda x}-1}{2\sqrt{\pi x^3}}
\frac1{\eta G_0(\eta)}e^{-\frac{\xi^2}{4x}}
\biggl(e^{-\frac{\xi\eta}{\sqrt2 b}}-e^{-\frac{\xi\eta}{\sqrt2 a}}
\biggr) d\xi dx
\]
is bounded by a constant multiple of $1/\eta G_0(\eta)$,
which is integrable on $(0,\infty)$, and \eqref{3.17} converges
as $\varepsilon\to0$.


\section{The expected volume of the Wiener sausage}

\noindent
Let $r$ be a given positive number.
The Wiener sausage $\{W(t)\}_{t\geqq0}$
for the Brownian motion with radius $r$ is defined by
\[
W(t)=\{x\in{\mathbb R}^d\,;\,x+B(s)\in U\text{ for some $s\in[0,t]$}\}
\]
for $t\geqq0$, where $\{B(t)\}_{t\geqq0}$ is
a Brownian motion on ${\mathbb R}^d$ and $U$ is the closed ball
with center $0$ and radius $r$.
For $t>0$ let
\[
L(t)=\int_{{\mathbb R}^d\setminus U}P_x[\tau\leqq t]dx,
\]
where $\tau=\inf\{t\geqq0\,;\,B(t)\in U\}$
and $P_x$ is the probability measure of events related to
the Brownian motion starting from $x\in{\mathbb R}^d$.
It is easy to see that the expectation of the volume of $W(t)$
coincides with the sum of $L(t)$ and the volume of $U$.

In the case when $d$ is odd, the explicit form of $L(t)$
has already given. One and three dimensional cases are easy.
Indeed, we have that
\[
L(t)=
\begin{cases}
2\sqrt{2t/\pi}\quad&\text{if $d=1$,}\\
2\pi rt+4r^2\sqrt{2\pi t}\quad&\text{if $d=3$.}
\end{cases}
\]
These formulae can be obtained directly from
the well-known formula for\linebreak $P_x[\tau\leqq t]$.
For details, see \cite{5,11,14}.
In the higher dimensional cases, the authors~\cite{7}
recently obtained an explicit form of $P_x[\tau\leqq t]$.
However it is not of a convenient form for the integration on $x$.

We here consider the Laplace transform of $L$ given by
\begin{equation}
\int_0^\infty e^{-\lambda t}L(t)dt=\frac{S_{d-1}r^{d-1}}{\sqrt{2\lambda^3}}
\frac{K_{d/2}(r\sqrt{2\lambda})}{K_{d/2-1}(r\sqrt{2\lambda})},
\quad \lambda>0,
\label{4.1}
\end{equation}
where $S_{d-1}$ is the surface area of $d-1$ dimensional unit sphere
(cf. \cite{5}).
When $d$ is odd, since $K_{d/2}(x)/K_{d/2-1}(x)$
may be expressed by the ratio of polynomials for $x>0$,
and the right hand side of \eqref{4.1} may be represented by
the linear combination of rational functions of
the following four types:
\[
\frac1{\sqrt\lambda},\,\,\frac1\lambda,\,\,
\frac1{\sqrt{\lambda^3}},\,\,\frac1{\sqrt\lambda-z}.
\]
Hence the Laplace transform on \eqref{4.1} can be inverted.
When $d$ is odd and more than or equal to five,
Theorem 1.1 in \cite{6} shows that, for $t>0$
\begin{equation}
L(t)=S_{d-1}r^{d-2}
\biggl[\frac{(d-2)t}2+\frac{r^2}{d-4}
-\frac{\sqrt2 r^3}{\sqrt{\pi t}}\sum_{j=1}^{N_d}
\frac1{(z_j^{(d)})^2}\int_0^\infty 
e^{-\frac{r^2x^2}{2t}+z_j^{(d)} x}dx\biggr].
\label{4.2}
\end{equation}
Here we have used $z_j^{(d)}$ and $N_d$
instead of $z_{d/2-1,j}$ and $N(d/2-1)$, respectively.

Our goal in this section is to give similar results
in even dimensional cases by applying the results in
Theorem \ref{Theorem 2.1}.

\medskip

\begin{thm}\label{Theorem 4.1}
For $x>0$ we set $G^{(d)}(x)=G_{d/2-1}(x)$.

\noindent
{\rm (1)} If $d=2$, we have
\[
L(t)=2\pi r\biggl[\sqrt{\frac{2t}\pi}+\frac{\sqrt 2 r^2}{\sqrt{\pi t}}
\int_0^\infty \!\!\! \int_0^\infty \frac{xy-1+e^{-xy}}{y^3 G^{(2)}(y)}
e^{-\frac{r^2x^2}{2t}}dxdy\biggr].
\]
{\rm (2)} If $d=4$, we have
\[
L(t)=2\pi^2r^2\biggl[t+\frac{\sqrt 2 r^3}{\sqrt{\pi t}}
\int_0^\infty \!\!\! \int_0^\infty \frac{1-e^{-xy}}{y^3 G^{(4)}(y)}
e^{-\frac{r^2x^2}{2t}}dxdy\biggr].
\]
{\rm (3)} If $d\geqq6$ and $d$ is even, we have
\[
\begin{split}
L(t)=S_{d-1}r^{d-2}\biggl[&\frac{(d-2)t}2+\frac{r^2}{d-4}
-\frac{\sqrt 2 r^3}{\sqrt{\pi t}}\sum_{j=1}^{N_d}
\frac1{(z_j^{(d)})^2}\int_0^\infty e^{-\frac{r^2x^2}{2t}+z_j^{(d)}x}dx\\
&+\frac{(-1)^{d/2-1}\sqrt 2 r^3}{\sqrt{\pi t}}
\int_0^\infty \!\!\! \int_0^\infty \frac{e^{-xy}}{y^3 G^{(d)}(y)}
e^{-\frac{r^2x^2}{2t}}dxdy\biggr].
\end{split}
\]
\end{thm}

\medskip

We set $T(t)=L(2r^2t)$ for $t>0$ and then deduce from \eqref{2.1} that,
for $\lambda>0$
\[
\int_0^\infty e^{-\lambda t}T(t)dt=
\frac{S_{d-1}r^d}{\sqrt{\lambda^3}}
\frac{K_{d/2}(\sqrt\lambda)}{K_{d/2-1}(\sqrt\lambda)}.
\]
For a proof of Theorem \ref{Theorem 4.1} we consider the function:
\[
\varSigma_\nu(\lambda)=\frac1{\sqrt{\lambda^3}}
\frac{K_{\nu+1}(\sqrt\lambda)}{K_\nu(\sqrt\lambda)},\quad \lambda>0.
\]
Let $T_\nu$ be the inverse Laplace transform of $\varSigma_\nu$.
Then we have the following, which immediately yields
Theorem \ref{Theorem 4.1}.

\medskip

\begin{thm}\label{Theorem 4.2}
Setting $\nu^+=\max\{\nu,0\}$, we have the following.

\noindent
{\rm (1)} If $|\nu|<1/2$,
\begin{equation}
T_\nu(t)=2\nu^+ t+2\sqrt{\frac t\pi}+\frac{\cos(\pi\nu)}{\sqrt{\pi t}}
\int_0^\infty \!\!\! \int_0^\infty
\frac{xy-1+e^{-xy}}{y^3 G_{|\nu|}(y)}e^{-\frac{x^2}{4t}}dxdy.
\label{4.3}
\end{equation}
{\rm (2)} If $1/2<|\nu|\leqq1$,
\begin{equation}
T_\nu(t)=2\nu^+ t-\frac{\cos(\pi\nu)}{\sqrt{\pi t}}
\int_0^\infty \!\!\! \int_0^\infty
\frac{1-e^{-xy}}{y^3 G_{|\nu|}(y)}e^{-\frac{x^2}{4t}}dxdy.
\label{4.4}
\end{equation}
{\rm (3)} If $1<|\nu|<3/2$,
\begin{equation}
T_\nu(t)=2\nu^+ t+\frac1{2(|\nu|-1)}+\frac{\cos(\pi\nu)}{\sqrt{\pi t}}
\int_0^\infty \!\!\! \int_0^\infty
\frac{e^{-xy}}{y^3 G_{|\nu|}(y)}e^{-\frac{x^2}{4t}}dxdy.
\label{4.5}
\end{equation}
{\rm (4)} If $|\nu|>3/2$ and $\nu-1/2$ is not an integer,
\begin{equation}
\begin{split}
T_\nu(t)=2\nu^+ t
&+\frac1{2(|\nu|-1)}
-\frac1{\sqrt{\pi t}}\sum_{j=1}^{N(\nu)}\frac1{z_{\nu,j}^2}
\int_0^\infty e^{-\frac{x^2}{4t}+z_{\nu,j}x}dx\\
&+\frac{\cos(\pi\nu)}{\sqrt{\pi t}}
\int_0^\infty \!\!\! \int_0^\infty
\frac{e^{-xy}}{y^3 G_{|\nu|}(y)}e^{-\frac{x^2}{4t}}dxdy.
\end{split}
\label{4.6}
\end{equation}
\end{thm}

\medskip

From now on, we shall treat $K_{\nu+1}/K_\nu$
as the function on the half real line $(0,\infty)$.
For our purpose we need three lemmas.

\medskip

\begin{lemma}\label{Lemma 4.3}
We have that, as $x\downarrow0$,
\[
\frac{K_{\nu+1}(x)}{K_\nu(x)}=
\begin{cases}
\dfrac{2\nu^+}x+o(1)&\text{if $\dfrac12<|\nu|\leqq1$,}\\
\dfrac{2\nu^+}x+\dfrac x{2(|\nu|-1)}+o(x)&\text{if $1<|\nu|<\dfrac32$,}\\
\dfrac{2\nu^+}x+\dfrac x{2(|\nu|-1)}+o(x^2)\quad
&\text{if $|\nu|>\dfrac32$, $\nu-\dfrac12\not\in\mathbb Z$.}
\end{cases}
\]
\end{lemma}

\medskip

In order to prove Lemma \ref{Lemma 4.3}, we need to show
an asymptotic behavior of $K_\mu(x)$ as $x\downarrow0$.

\medskip

\begin{lemma}\label{Lemma 4.4}
If $n+1/2<\mu<n+3/2$ for an integer $n\geqq1$, we have
\[
K_\mu(x)
=\frac{\varGamma(\mu)}2\biggl(\frac 2x\biggr)^\mu e^{-x}
\biggl\{ 1+x+\sum_{k=2}^{2n+1}\binom{\mu-\frac12}k
\frac{\varGamma(2\mu-k)}{\varGamma(2\mu)}(2x)^k+o(x^{2n+1})\biggr\}
\]
as $x\downarrow0$. Especially, we have
\begin{equation}
K_\mu(x)=
\frac{\varGamma(\mu)}2\biggl(\frac 2x\biggr)^\mu e^{-x}
\biggl\{ 1+x+\frac12\frac{2\mu-3}{2\mu-2}x^2
+\frac16\frac{2\mu-5}{2\mu-2}x^3+o(x^3)\biggr\}.
\label{4.7}
\end{equation}
\end{lemma}
\begin{proof}
It is well-known that, for $\mu>-1/2$
\[
\begin{split}
K_\mu(x)
&=\sqrt{\frac\pi{2x}}\frac{e^{-x}}{\varGamma(\mu+1/2)}
\int_0^\infty e^{-y}y^{\mu-1/2}\biggl(1+\frac y{2x}\biggr)^{\mu-1/2}dy\\
&=\frac{\sqrt \pi}{(2x)^\mu}\frac{e^{-x}}{\varGamma(\mu+1/2)}
\int_0^\infty e^{-y}y^{2\mu-1}\biggl(1+\frac{2x}y\biggr)^{\mu-1/2}dy\\
\end{split}
\]
(cf. \cite[p.140]{13}, \cite[p.206]{19}). The Taylor formula yields that
\[
\begin{split}
&\int_0^\infty e^{-y}y^{2\mu-1}\biggl(1+\frac{2x}y\biggr)^{\mu-1/2}dy\\
&=\int_0^\infty e^{-y}y^{2\mu-1} \biggl\{1+\sum_{k=1}^{2n}
\binom{\mu-\frac12}k
\biggl(\frac{2x}y\biggr)^k\\
&\hphantom{=\int_0^\infty e^{-y}y^{2\mu-1} \biggl\{1\,\,}
+\binom{\mu-\frac12}{2n+1}\biggl(\frac{2x}y\biggr)^{2n+1}
\biggl(1+\frac{2\xi x}y\biggr)^{\mu-2n-3/2}\biggr\}dy
\end{split}
\]
for some $\xi\in[0,1]$.
Since $2\mu-2n-2>-1$. Since $\mu-2n-3/2<0$, we have
\[
\begin{split}
&\int_0^\infty e^{-y}y^{2\mu-2n-2}
\biggl(1+\frac{2x}y\biggr)^{\mu-2n-3/2}dy\\
&\leqq\int_0^\infty e^{-y}y^{2\mu-2n-2}
\biggl(1+\frac{2\xi x}y\biggr)^{\mu-2n-3/2}dy\\
&\leqq\int_0^\infty e^{-y}y^{2\mu-2n-2}dy=\varGamma(2\mu-2n-1).
\end{split}
\]
Hence the dominated convergence theorem yields that
\[
\lim_{x\downarrow0}\int_0^\infty e^{-y}y^{2\mu-2n-2}
\biggl(1+\frac{2x}y\biggr)^{\mu-2n-3/2}dy
=\varGamma(2\mu-2n-1).
\]
Therefore we obtain
\[
\begin{split}
K_\mu(x)=\frac{\sqrt \pi}{(2x)^\mu}\frac{e^{-x}}{\varGamma(\mu+1/2)}
\biggl[&\varGamma(2\mu)+(2\mu-1)\varGamma(2\mu-1)x\\
&+\sum_{k=2}^{2n}\binom{\mu-\frac12}k \varGamma(2\mu-k)(2x)^k\\
&+\binom{\mu-\frac12}{2n+1}\varGamma(2\mu-2n-1)(2x)^{2n+1}\{1+o(1)\} \biggr]
\end{split}
\]
as $x\downarrow0$. With the help of the formula
\begin{equation}
2^{2z-1}\varGamma(z)\varGamma\biggl(z+\frac12\biggr)=\sqrt\pi \varGamma(2z)
\label{4.8}
\end{equation}
(cf. \cite[p.3]{13}), we easily obtain the assertion.
\end{proof}

\medskip

\begin{proof}[Proof of Lemma \ref{Lemma 4.3}]
We only consider the case of $\nu\geqq0$.
If we show this case, the result for $\nu<0$ follows from \eqref{2.8}.

If $\nu>3/2$ and $\nu-1/2$ is not an integer,
it follows from \eqref{4.7} that, as $x\downarrow0$,
\[
\frac{K_{\nu+1}(x)}{K_\nu(x)}
=\frac{2\nu}x+\frac x{2(\nu-1)}+o(x^2).
\]
When $1<\mu<3/2$, we use the formula
\[
\begin{split}
&\int_0^\infty e^{-y}y^{2\mu-1}\biggl(1+\frac{2x}y\biggr)^{\mu-1/2}dy\\
&=\int_0^\infty e^{-y}y^{2\mu-1} \biggl\{
1+\binom{\mu-\frac12}1\frac{2x}y+\binom{\mu-\frac12}2
\biggl(\frac{2x}y\biggr)^2
\biggl(1+\frac{2\xi x}y\biggr)^{\mu-5/2}\biggr\}dy
\end{split}
\]
in a similar way to Lemma \ref{Lemma 4.4}. Then we have
\[
K_\mu(x)=\frac{\varGamma(\mu)}2\biggl(\frac 2x\biggr)^\mu e^{-x}
\biggl\{ 1+x+\frac12\frac{2\mu-3}{2\mu-2}x^2+o(x^2)\biggr\}.
\]
Combining it with \eqref{4.7}, we deduce that, if $1<\nu<3/2$,
\[
\frac{K_{\nu+1}(x)}{K_\nu(x)}
=\frac{2\nu}x+\frac x{2(\nu-1)}+o(x).
\]

In the case of $1/2<\mu\leqq1$, the calculation is simpler.
Indeed, by the same calculation as \eqref{4.7}, we can easily obtain
\begin{equation}
K_\mu(x)=\frac{\varGamma(\mu)}2\biggl(\frac 2x\biggr)^\mu e^{-x}
\{ 1+x+o(x)\}
\label{4.9}
\end{equation}
in virtue of the formula
\[
\begin{split}
&\int_0^\infty e^{-y}y^{2\mu-1}\biggl(1+\frac{2x}y\biggr)^{\mu-1/2}dy\\
&=\int_0^\infty e^{-y}y^{2\mu-1} \biggl\{
1+\binom{\mu-\frac12}1\frac{2x}y
\biggl(1+\frac{2\xi x}y\biggr)^{\mu-3/2}\biggr\}dy.
\end{split}
\]
Hence we deduce from \eqref{4.7} and \eqref{4.9} that
\[
\frac{K_{\nu+1}(x)}{K_\nu(x)}
=\frac{2\nu}x \frac{1+x+\dfrac12\dfrac{2\nu-1}{2\nu}x^2+o(x^2)}{1+x+o(x)}
=\frac{2\nu}x+o(1).
\]
We finish the proof of Lemma \ref{Lemma 4.3}.
\end{proof}

\medskip

For an integer $k\geqq1$ we set
\[
\zeta_{\nu,k}=\sum_{j=1}^{N(\nu)}\frac1{z_{\nu,j}^k},\quad
\varrho_{\nu,k}=\int_0^\infty \frac{dy}{y^{k+1}G_{|\nu|}(y)}.
\]
In virtue of Theorem \ref{Theorem 2.1} and Lemma \ref{Lemma 4.3},
we can derive the connection between $\zeta_{\nu,k}$ and $\varrho_{\nu,k}$.

\medskip

\begin{lemma}\label{Lemma 4.5}
{\rm (1)} If $1/2<|\nu|\leqq1$, we have
\begin{equation}
1+\varrho_{\nu,1}\cos(\pi\nu)=0.
\label{4.10}
\end{equation}
{\rm (2)} If $1<|\nu|<3/2$, we have
\begin{equation}
\begin{cases}
1+\varrho_{\nu,1}\cos(\pi\nu)=0,\\
-\varrho_{\nu,2}\cos(\pi\nu)=\dfrac1{2(|\nu|-1)}.
\end{cases}
\label{4.11}
\end{equation}
{\rm (3)} If $|\nu|>3/2$ and $\nu-1/2$ is not an integer,
we have
\begin{equation}
\begin{cases}
1+\zeta_{\nu,1}+\varrho_{\nu,1}\cos(\pi\nu)=0,\\
\zeta_{\nu,2}-\varrho_{\nu,2}\cos(\pi\nu)=\dfrac1{2(|\nu|-1)}.
\end{cases}
\label{4.12}
\end{equation}
\end{lemma}
\begin{proof}
Recall that, for $\nu\neq0$
\begin{equation}
G_{|\nu|}(x)=
\begin{cases}
\dfrac\pi{2x}e^{2x}\{1+o(1)\}\quad&\text{as $x\to\infty$,}\\
\dfrac1{\kappa_{|\nu|} x^{2|\nu|}}\{1+o(1)\}\quad&\text{as $x\downarrow0$}.
\end{cases}
\label{4.13}
\end{equation}
If $1/2<|\nu|\leqq1$, we obtain by \eqref{4.13} that $1/y^2 G_{|\nu|}(y)$
is asymptotically equal to $\kappa_{|\nu|}y^{2|\nu|-2}$ as $y\downarrow0$.
This implies the convergence of $\varrho_{\nu,1}$.
The dominated convergence theorem shows that, as $x\downarrow0$,
\[
\int_0^\infty \frac{dy}{(y+x)y G_{|\nu|}(y)}=\varrho_{\nu,1}+o(1).
\]
Hence we deduce from \eqref{2.5} that
\[
\frac{K_{\nu+1}(x)}{K_\nu(x)}=1+\frac{2\nu^+}x
+\cos(\pi\nu)\varrho_{\nu,1}+o(1).
\]
With the help of Lemma \ref{Lemma 4.3}, we conclude \eqref{4.10}.

In the case of $1<|\nu|<3/2$,  $1/y^3 G_{|\nu|}(y)$ is integrable
by \eqref{4.13} and we have
\begin{equation}
\begin{split}
\int_0^\infty \frac{dy}{(y+x)y G_{|\nu|}(y)}
&=\int_0^\infty\biggl\{\frac1y-\frac x{y(y+x)}\biggr\}
\frac{dy}{y G_{|\nu|}(y)}\\
&=\varrho_{\nu,1}-x\int_0^\infty \frac{dy}{(y+x)y^2 G_{|\nu|}(y)}\\
&=\varrho_{\nu,1}-x\varrho_{\nu,2}+o(x)
\end{split}
\label{4.14}
\end{equation}
by the dominated convergence theorem. Then we get by \eqref{2.5} that
\[
\frac{K_{\nu+1}(x)}{K_\nu(x)}=\frac{2\nu^+}x
+1+\varrho_{\nu,1}\cos(\pi\nu)
-x\varrho_{\nu,2}\cos(\pi\nu)+o(x),
\]
and hence, combining with Lemma \ref{Lemma 4.3}, we get \eqref{4.11}.

If $|\nu|>3/2$ and $\nu-1/2$ is not an integer,
we deduce
\[
\begin{split}
\sum_{j=1}^{N(\nu)}\frac1{z_{\nu,j}-x}
&=\sum_{j=1}^{N(\nu)}\biggl\{\frac1{z_{\nu,j}}
+\frac x{z_{\nu,j}(z_{\nu,j}-x)}\biggr\}\\
&=\zeta_{\nu,1}+x\sum_{j=1}^{N(\nu)}\frac1{z_{\nu,j}(z_{\nu,j}-x)}\\
&=\zeta_{\nu,1}+x\zeta_{\nu,2}+o(x).
\end{split}
\]
Since $1/y^3 G_{|\nu|}(y)$ is integrable on $(0,\infty)$
by \eqref{4.13} and \eqref{4.14} is valid for this case.
Therefore it follows from \eqref{2.6} that
\[
\frac{K_{\nu+1}(x)}{K_\nu(x)}=\frac{2\nu^+}x
+1+\zeta_{\nu,1}+\varrho_{\nu,1}\cos(\pi\nu)
+x\{\zeta_{\nu,2}-\varrho_{\nu,2}\cos(\pi\nu)\}+o(x).
\]
We immediately obtain \eqref{4.12} by Lemma \ref{Lemma 4.3}.
\end{proof}

\medskip

\begin{remark}\label{Remark 4.6}
In the case when $|\nu|>3/2$ and $\nu-1/2\notin\mathbb Z$,
we can derive
$$
\zeta_{\nu,3}+\varrho_{\nu,3}\cos(\pi\nu)=0.
$$
It is not necessary for the proof of Theorem \ref{Theorem 4.2}.
\end{remark}

\medskip

In virtue of Theorem \ref{Theorem 2.1},
we have that, if $|\nu|<3/2$ and $\nu-1/2\notin \mathbb Z$,
\begin{equation}
\varSigma_\nu(\lambda)=\frac{2\nu^+}{\lambda^2}
+\frac1{\sqrt{\lambda^3}}+\cos(\pi\nu)
\int_0^\infty \frac{dy}{\sqrt{\lambda^3}(\sqrt\lambda+y)yG_{|\nu|}(y)}
\label{4.15}
\end{equation}
and that, if $|\nu|>3/2$ and $\nu-1/2\notin \mathbb Z$,
\begin{equation}
\begin{split}
\varSigma_\nu(\lambda)=\frac{2\nu^+}{\lambda^2}
&+\frac1{\sqrt{\lambda^3}}
+\sum_{j=1}^{N(\nu)}\frac1{\sqrt{\lambda^3}(z_{\nu,j}-\sqrt\lambda)}\\
&+\cos(\pi\nu)\int_0^\infty
\frac{dy}{\sqrt{\lambda^3}(\sqrt\lambda+y)yG_{|\nu|}(y)}.
\end{split}
\label{4.16}
\end{equation}

\medskip

\begin{lemma}\label{Lemma 4.7}
For $t>0$ let
\[
q_\nu(t)=\int_0^\infty \!\!\! \int_0^\infty
\frac{xy-1+e^{-xy}}{y^3 G_{|\nu|}(y)}p(t,x)dxdy,
\]
where
\[
p(t,x)=\frac1{\sqrt{\pi t}}e^{-\frac{x^2}{4t}}.
\]
Then, we have, for $\lambda>0$
\begin{equation}
\int_0^\infty e^{-\lambda t}q_\nu(t)dt=\int_0^\infty
\frac{dy}{\sqrt{\lambda^3}(\sqrt\lambda+y)yG_{|\nu|}(y)}.
\label{4.17}
\end{equation}
\end{lemma}
\begin{proof}
We first recall the elementary formula
\begin{equation}
\int_0^\infty e^{-\lambda t}p(t,x)dt
=\frac1{\sqrt\lambda}e^{-\sqrt\lambda x},\quad \lambda>0.
\label{4.18}
\end{equation}
Then, with the help of the Fubini theorem, we deduce
\[
\int_0^\infty e^{-\lambda t} q_\nu(t)dt
=\int_0^\infty dy\int_0^\infty \frac1{\sqrt\lambda}e^{-\sqrt\lambda x}
\frac{xy-1+e^{-xy}}{y^3 G_{|\nu|}(y)}dx.
\]
Carrying out the elementary integral in $x$, we obtain \eqref{4.17}.
\end{proof}

\medskip

We now complete our proof of Theorem \ref{Theorem 4.2}.
If $|\nu|<3/2$ and $\nu-1/2\notin\mathbb Z$,
\begin{equation}
T_\nu(t)=2\nu^+ t+2\sqrt{\frac t\pi}+\cos(\pi\nu)q_\nu(t).
\label{4.19}
\end{equation}
When $|\nu|<1/2$, \eqref{4.19} immediately implies \eqref{4.3}.

When $|\nu|>1/2$ and $\nu-1/2\notin\mathbb Z$,
we have by \eqref{4.13} that $1/y^2 G_{|\nu|}(y)$ is integrable
on $(0,\infty)$. Since
\[
\biggl| \frac{1-e^{-xy}}{y^3 G_{|\nu|}(y)}p(t,x) \biggr|
\leqq\frac{x p(t,x)}{y^2 G_{|\nu|}(y)},
\]
we have
\begin{equation}
q_\nu(t)
=2\sqrt{\frac t\pi}\varrho_{\nu,1}
-\int_0^\infty \!\!\! \int_0^\infty
\frac{1-e^{-xy}}{y^3 G_{|\nu|}(y)}p(t,x)dxdy.
\label{4.20}
\end{equation}
Combining this formula with \eqref{4.10} and \eqref{4.19},
we conclude \eqref{4.4}.

When $|\nu|>1$ and $\nu-1/2\notin\mathbb Z$,
we can further improve \eqref{4.20}.
Indeed, since $1/y^3 G_{|\nu|}(y)$ is integrable on $(0,\infty)$,
we get
\[
\int_0^\infty \int_0^\infty \frac{e^{-xy}}{y^3 G_{|\nu|}(y)}p(t,x)dxdy
\]
converges and we get
\begin{equation}
q_\nu(t)=2\sqrt{\frac t\pi}\varrho_{\nu,1}-\varrho_{\nu,2}
+\int_0^\infty \!\!\! \int_0^\infty
\frac{e^{-xy}}{y^3 G_{|\nu|}(y)}p(t,x)dxdy.
\label{4.21}
\end{equation}
Hence, we get \eqref{4.5} by \eqref{4.19},
\eqref{4.21} and Lemma \ref{Lemma 4.3}.

To invert $\varSigma_\nu$ in the case of $|\nu|>3/2$,
we note
\[
\frac1{a\pm b}=\frac1a\mp\frac b{a^2}+\frac{b^2}{a^2(a\pm b)}.
\]
Then it follows that
\[
\sum_{j=1}^{N(\nu)}\frac1{z_{\nu,j}-\sqrt\lambda}
=\zeta_{\nu,1}+\sqrt\lambda \zeta_{\nu,2}
+\sum_{j=1}^{N(\nu)}\frac{\lambda}{z_{\nu,j}^2(z_{\nu,j}-\sqrt\lambda)}.
\]
Hence \eqref{4.16} is equivalent to
\[
\begin{split}
\varSigma_\nu(\lambda)=\frac{2\nu^+}{\lambda^2}
&+\frac{1+\zeta_{\nu,1}}{\sqrt{\lambda^3}}
+\frac{\zeta_{\nu,2}}{\lambda}
-\sum_{j=1}^{N(\nu)}
\frac1{z_{\nu,j}^2\sqrt\lambda(\sqrt\lambda-z_{\nu,j})}\\
&+\cos(\pi\nu)\int_0^\infty
\frac{dy}{\sqrt{\lambda^3}(\sqrt\lambda+y)yG_{|\nu|}(y)}.
\end{split}
\]
With the help of \eqref{4.18}, we easily see that, for $z\in\mathbb C^-$
\[
\int_0^\infty e^{-\lambda t} dt \int_0^\infty e^{zx}p(t,x)dx
=\frac1{\sqrt\lambda(\sqrt\lambda-z)}.
\]
Hence, we deduce \eqref{4.6} from \eqref{4.21} and Lemma \ref{Lemma 4.3}.


\section{Large time asymptotics of the Wiener sausage}

\noindent
This section is devoted to show an asymptotic behavior
of $L(t)$ for large $t$ in even dimensional cases.
Le~Gall~\cite{14} considered the Wiener sausage
associated with a general compact set and proved
\begin{equation}
L(t)=
\begin{cases}
c_1^{(4)}t+c_2^{(4)}\log t+c_3^{(4)}
+c_4^{(4)}\dfrac{\log t}t+o\biggl(\dfrac{\log t}t\biggr)
&\quad\text{if $d=4$,}\\
c_1^{(d)}t+c_2^{(d)}+c_3^{(d)}t^{2-d/2}
+O(t^{1-d/2})&\quad\text{if $d\geqq5$}
\end{cases}
\label{5.1}
\end{equation}
and gave the explicit expression of each constant $c_j^{(d)}$.
In the two dimensional case, Le~Gall~\cite{15} also showed
that $L(t)$ admits the asymptotic expansion in powers of $1/\log t$.

When $d$ is odd, Hamana~\cite{6} showed that the asymptotic expansion
for $L(t)$ with the help of \eqref{4.2}.
The purpose in this section is to improve the asymptotic behavior
of $L(t)$ if $d$ is even and not less than 6.
Throughout this section, we use $C_i$'s for 
positive constants independent of the variable.

\medskip

\begin{thm}\label{Theorem 5.1}
If $d$ is even and not less than 6,
there is a family of constants $\{\alpha_n^{(d)}\}_{n=0}^{d-5}$
such that
\[
\begin{split}
L(t)=S_{d-1}r^{d-2}\biggl[ &\frac{(d-2)t}2 +\frac{r^2}{d-4}
-\frac{\alpha_d \, r^{d-2}}{2^{d/2-2}(d-4)\varGamma(d/2-1)}\frac1{t^{d/2-2}}\\
&+\frac1{t^{d/2-1}}\sum_{n=0}^{d-5}\frac{\alpha_n^{(d)}}{t^{n/2}}
+\frac{\varGamma((d-3)/2)r^{2d-4}}
{\sqrt\pi(d-2)\varGamma(d/2-1)^3}\frac{\log t}{t^{d-3}}
+O\biggl(\frac1{t^{d-3}}\biggr)\biggr],
\end{split}
\]
where
\[
\alpha_d=\int_0^\infty \frac1{y^d}
\biggl\{\frac1{G^{(d)}(y)}-\frac{y^{d-2}}{2^{d-4}\varGamma(d/2-1)^2}\biggr\}dy.
\]
\end{thm}

\medskip

Recalling the notation $G^{(d)}=G_{d/2-1}$, we set
\[
\begin{split}
&L_1(t)=\frac{\sqrt 2 r^3}{\sqrt{\pi t}}\sum_{j=1}^{N_d}
\frac1{(z_j^{(d)})^2}\int_0^\infty e^{-\frac{r^2x^2}{2t}+z_j^{(d)}x}dx\\
&L_2(t)=\frac{\sqrt 2 r^3}{\sqrt{\pi t}}
\int_0^\infty \!\!\! \int_0^\infty \frac{e^{-xy}}{y^3 G^{(d)}(y)}
e^{-\frac{r^2x^2}{2t}}dxdy.
\end{split}
\]
Then we have from Theorem \ref{Theorem 4.1} that,
if $d$ is even and not less than 6,
\begin{equation}
L(t)=S_{d-1}r^{d-2}\biggl[\frac{(d-2)t}2+\frac{r^2}{d-4}
-L_1(t)+(-1)^{d/2-1}L_2(t)\biggr].
\label{5.2}
\end{equation}
The calculation of $L_1(t)$ is easy since $\RE(z_j^{(d)})<0$
for each $j=1,2,\dots,N_d$. For $x\geqq0$ and
an integer $n\geqq0$ we put
\[
R_n(x)=e^{-x}-\sum_{k=0}^n \frac{(-1)^k}{k!}x^k.
\]
and let $M$ be a positive integer. Then it follows that
\[
\begin{split}
\int_0^\infty e^{-\frac{r^2x^2}{2t}+z_j^{(d)}x}dx
&=\sum_{n=0}^M \frac{(-1)^nr^{2n}}{n!\,(2t)^n}
\int_0^\infty x^{2n}e^{z_j^{(d)}x}dx
+\int_0^\infty R_M\biggl(\frac{r^2x^2}{2t}\biggr)e^{z_j^{(d)}x}dx\\
&=\sum_{n=0}^M \frac{(-1)^nr^{2n}(2n)!}{2^n n!\,(z_j^{(d)})^{2n+1}}
\frac1{t^n}+O\biggl(\frac1{t^{M+1}}\biggr).
\end{split}
\]
Hence we obtain
\begin{equation}
L_1(t)=\sqrt{\frac 2 \pi} \sum_{n=0}^M
\frac{(-1)^nr^{2n+3}(2n)!\,\zeta_{2n+3}^{(d)}}{2^n n!}\frac1{t^{n+1/2}}
+O\biggl(\frac1{t^{M+3/2}}\biggr)
\label{5.3}
\end{equation}
as $t\to\infty$, where the notation $\zeta_k^{(d)}$
is used to denote $\zeta_{d/2-1,k}$ for an integer $k\geqq1$.

For a proof of Theorem \ref{Theorem 5.1}
we need to give an asymptotic behavior
of $L_2(t)$ for large $t$. Let $m=d/2-1$ for simplicity and
set $L_2^0(t)=L_2(2r^2 t)/r^2$. We have
\begin{equation}
L_2^0(t)=\frac1{\sqrt{\pi t}}\int_0^\infty \int_0^\infty
\frac{e^{-xy}}{y^3 G_m(y)}e^{-\frac{x^2}{4t}}dxdy.
\label{5.4}
\end{equation}
We note that $G_m(x)=K_m(x)^2+\pi^2 I_m(x)^2$.

\medskip

\begin{lemma}\label{Lemma 5.2}
For an integer $n\geqq1$ we have
\begin{equation}
\int_0^\infty \frac{R_{n-1}(x^2)}{x^{2n}}dx
=\frac{(-1)^n\pi}{2\varGamma(n+1/2)}.
\label{5.5}
\end{equation}
\end{lemma}
\begin{proof}
From a change of variables from $x$ to $y$ given by $y=x^2$
we deduce that the left hand side of \eqref{5.5} is equal to
\[
\frac12\int_0^\infty \frac1{y^{n+1/2}}\biggl\{ e^{-y}
-\sum_{k=0}^{n-1}\frac{(-1)^k}{k!}y^k\biggr\}dy
=\frac12\varGamma\biggl(\frac12-n\biggr)
\]
(cf. \cite[p.361]{3}). By the formula
$\varGamma(z)\varGamma(1-z)=\pi/\sin(\pi z)$,
$z\in\mathbb C\setminus\mathbb Z$
(cf. \cite[p.3]{13}), we have
\[
\varGamma\biggl(\frac12-n\biggr)=\frac{(-1)^n\pi}{\varGamma(n+1/2)},
\]
which yields \eqref{5.5}.
\end{proof}

\medskip

We give several constants which we need to describe
the asymptotic behavior of $1/G_m(x)$ as $x\downarrow0$.
For integers $h,k$ with $1\leqq h\leq k\leqq m-1$ we put
\[
b_{k,h}= \sum_{\scriptstyle k_1+k_2+\cdots+k_h=k
\atop\scriptstyle k_1,k_2,\dots,k_h\geqq1}
b_{k_1}b_{k_2}\cdots b_{k_h},
\]
where
\[
b_k=\frac{(-1)^{k+1}}{4^k\varGamma(m)^2} \sum_{h=0}^k
\frac{\varGamma(m-h)\varGamma(m-k+h)}{\varGamma(h+1)\varGamma(k-h+1)}.
\]
We set
\[
a_k=\sum_{h=1}^k b_{k,h}\,\, (k=1,2,\dots,m-1),
\quad
a_m=\frac{(-1)^{m+1}}{4^{m-1}m\varGamma(m)^2}.
\]
We note that $a_1=b_1$.
Moreover recall that $\kappa_m=1/4^{m-1}\varGamma(m)^2$.
See \eqref{2.25} for the definition of $\kappa_\nu$.
The second lemma gives the asymptotic behavior of $1/G_m(x)$.

\medskip

\begin{lemma}\label{Lemma 5.3}
We have that, as $x\downarrow0$,
\begin{equation}
\frac1{G_m(x)}=\kappa_m x^{2m}
\biggl\{1+\sum_{k=1}^{m-1}a_k x^{2k}+a_m x^{2m}\log\frac1x
+O(x^{2m})\biggr\}.
\label{5.6}
\end{equation}
\end{lemma}
\begin{proof} By the series expression of
the modified Bessel function $I_m$, we have that 
$I_m(x)^2$ is of order $x^{2m}$ as $x\downarrow0$. It is known that
\[
\begin{split}
K_m(x)=&\frac12\sum_{k=0}^{m-1}\frac{(-1)^k(m-k-1)!}{k!}
\biggl(\frac x2\biggr)^{2k-m}\\
&+(-1)^{m+1} \log\frac x2 \sum_{k=0}^\infty \frac1{k!\,(k+m)!}
\biggl(\frac x2\biggr)^{2k+m}\\
&-\frac{(-1)^{m+1}}2\sum_{k=0}^\infty \frac1{k!\,(k+m)!}
\biggl(\frac x2\biggr)^{2k+m}\{\psi(k+1)+\psi(k+m+1)\},
\end{split}
\]
where $\psi$ is the logarithmic derivative of the gamma function
(cf. \cite[p.80]{19}). This formula immediately yields
\[
\begin{split}
K_m(x)=&\frac{\varGamma(m)}2\biggl(\frac2x\biggr)^m
\sum_{k=0}^{m-1}\frac{(-1)^k \varGamma(m-k)}{\varGamma(k+1)\varGamma(m)}
\biggl(\frac x2\biggr)^{2k}\\
&+\frac{(-1)^m}{\varGamma(m+1)}\biggl(\frac x2\biggr)^m\log\frac1x
+O(x^m).
\end{split}
\]
as $x\downarrow0$. Hence we have
\[
\begin{split}
K_m(x)^2=
&\frac{\varGamma(m)^2}4\biggl(\frac2x\biggr)^{2m}
\biggl\{\sum_{k=0}^{m-1}
\frac{(-1)^k \varGamma(m-k)}{\varGamma(k+1)\varGamma(m)}
\biggl(\frac x2\biggr)^{2k}\biggr\}^2\\
&+\frac{(-1)^m}m\log\frac1x \sum_{k=0}^{m-1}
\frac{(-1)^k \varGamma(m-k)}{\varGamma(k+1)\varGamma(m)}
\biggl(\frac x2\biggr)^{2k}+O(1).
\end{split}
\]
Therefore we conclude that, as $x\downarrow0$,
\[
\begin{split}
G_m(x)&=K_m(x)^2+O(x^{2m})\\
&=\frac1{\kappa_m x^{2m}}
\biggl\{1-\sum_{k=1}^{m-1}b_k x^{2k}
-a_m x^{2m}\log\frac1x+O(x^{2m})\biggr\}.
\end{split}
\]

We set $G_m^0(x)=\kappa_m x^{2m}G_m(x)$ for simplicity.
It is sufficient to obtain the asymptotic behavior of $1/G_m^0(x)$.
We can easily derive
\begin{equation}
\frac1{G_m^0(x)}=
1+\sum_{h=1}^{m-1}\biggl(\sum_{k=1}^{m-1}b_k x^{2k}
+a_m x^{2m}\log\frac1x\biggr)^h+O(x^{2m}).
\label{5.7}
\end{equation}
In the case of $m=2$, \eqref{5.7} immediately implies \eqref{5.6}.
We concentrate on considering the case of $m\geqq3$.
In this case, the summation in the right hand side of
\eqref{5.7} is
\begin{equation}
1+\sum_{h=1}^{m-1}\biggl(\sum_{k=1}^{m-1} b_k x^{2k}\biggr)^h
+a_m x^{2m}\log\frac1x+O(x^{2m}).
\label{5.8}
\end{equation}
A simple calculation shows that the double sum
in the right hand side of \eqref{5.8} is equal to
\[
\sum_{h=1}^{m-1}\sum_{k=h}^{m-1} b_{k,h} x^{2k}+O(x^{2m}).
\]
Hence we deduce
\[
\frac1{G_m^0(x)}=1+\sum_{k=1}^{m-1}a_k x^{2k}+a_m x^{2m}\log\frac1x
+O(x^{2m}),
\]
which implies \eqref{5.6} for $m\geqq3$.
\end{proof}

\medskip

We now proceed to a proof of Theorem \ref{Theorem 5.1}.
From \eqref{5.4} it follows that
\[
L_2^0(t)=\frac2{\sqrt\pi}\int_0^\infty \!\!\! \int_0^\infty
\frac1{y^3 G_m(y)}e^{-2\sqrt t uy}e^{-u^2}dudy,
\]
which is the sum of
\[
\begin{split}
&L_2^1(t)=\frac2{\sqrt\pi}\sum_{n=0}^{m-2}\frac{(-1)^n}{n!}
\int_0^\infty \!\!\! \int_0^\infty
\frac1{y^3 G_m(y)}u^{2n}e^{-2\sqrt t uy}dudy,\\
&L_2^2(t)=\frac2{\sqrt\pi} \int_0^\infty \!\!\! \int_0^\infty
\frac1{y^3 G_m(y)}e^{-2\sqrt t uy} R_{m-2}(u^2) dudy.
\end{split}
\]
We easily derive
\[
\begin{split}
L_2^1(t)&=\frac1{\sqrt\pi}\sum_{n=0}^{m-2}
\frac{(-1)^n(2n)!}{4^n n!}\frac1{t^{n+1/2}}
\int_0^\infty \frac{dy}{y^{2n+4}G_m(y)}\\
&=\frac1{\sqrt\pi}\sum_{n=0}^{d/2-3}
\frac{(-1)^n(2n)!\,\varrho_{2n+3}^{(d)}}{4^n n!}\frac1{t^{n+1/2}},
\end{split}
\]
where $\varrho_k^{(d)}=\varrho_{m,k}$.

For $x>0$ and an integer $k=0,1,2,\dots,m-1$ let
\[
Q_k(x)=\frac1{G_m(x)}-\kappa_m x^{2m}\sum_{n=0}^k a_n x^{2n},
\]
where we have put $a_0=1$ for convenience.
We need the following lemma to derive
the large time asymptotics of $L_2^2(t)$.

\medskip

\begin{lemma}\label{Lemma 5.4}
We have that
$L_2^2(t)$ is the sum of the following three integrals;
\begin{align}
&\frac2{\sqrt\pi}\int_0^\infty \!\!\! \int_0^\infty
\frac{e^{-2\sqrt t uy}}{y^3}
\kappa_m y^{2m}\sum_{k=0}^{m-1}a_k y^{2k} R_{m+k-2}(u^2)dudy,
\label{5.9}\\
&\frac2{\sqrt\pi}\int_0^\infty \!\!\! \int_0^\infty
\frac{e^{-2\sqrt t uy}}{y^3}
\sum_{k=0}^{m-1}\frac{(-1)^{m+k-1} Q_k(y)}{(m+k-1)!}u^{2(m+k-1)}dudy,
\label{5.10}\\
&\frac2{\sqrt\pi}\int_0^\infty \!\!\! \int_0^\infty
\frac{e^{-2\sqrt t uy}}{y^3} Q_{m-1}(y) R_{2m-2}(u^2)dudy.
\label{5.11}
\end{align}
\end{lemma}
\begin{proof}
By the definition of $Q_0$ and $R_k$, we have
\begin{equation}
\begin{split}
L_2^2(t)=&\frac2{\sqrt\pi}\int_0^\infty \!\!\! \int_0^\infty
\frac{e^{-2\sqrt t uy}}{y^3}\kappa_m y^{2m}R_{m-2}(u^2)dudy\\
&+\frac2{\sqrt\pi}\int_0^\infty \!\!\! \int_0^\infty
\frac{e^{-2\sqrt t uy}}{y^3}
\frac{(-1)^{m-1}Q_0(y)}{(m-1)!}u^{2(m-1)}dudy\\
&+\frac2{\sqrt\pi}\int_0^\infty \!\!\! \int_0^\infty
\frac{e^{-2\sqrt t uy}}{y^3} Q_0(y) R_{m-1}(u^2)dudy.
\end{split}
\label{5.12}
\end{equation}
Moreover, it follows that, for $y>0$, $u>0$ and an integer $k\geqq1$,
\[
\begin{split}
Q_{k-1}(y) R_{m+k-2}(u^2)
=&Q_k(y) R_{m+k-1}(u^2)+\kappa_m a_k y^{2(m+k)} R_{m+k-2}(u^2)\\
&+\frac{(-1)^{m+k-1}}{(m+k-1)!}Q_k(y)u^{2(m+k-1)}.
\end{split}
\]
Taking the sum on $k$ over $[1,m-1]$, we deduce that the third term
of the right hand side of \eqref{5.12} is equal to
\[
\begin{split}
&\frac2{\sqrt\pi}\int_0^\infty \!\!\! \int_0^\infty
\frac{e^{-2\sqrt t uy}}{y^3} Q_{m-1}(y) R_{2m-2}(u^2)dudy\\
&+\frac2{\sqrt\pi}\int_0^\infty \!\!\! \int_0^\infty
\frac{e^{-2\sqrt t uy}}{y^3}
\kappa_m y^{2m} \sum_{k=1}^{m-1}a_ky^{2k} R_{m+k-2}(u^2)dudy\\
&+\frac2{\sqrt\pi}\int_0^\infty \!\!\! \int_0^\infty
\frac{e^{-2\sqrt t uy}}{y^3}
\sum_{k=1}^{m-1} \frac{(-1)^{m+k-1}Q_k(y)}{(m+k-1)!}u^{2(m+k-1)} dudy.
\end{split}
\]
Hence we conclude that $L_2^2(t)$ is the sum of
\eqref{5.9}, \eqref{5.10} and \eqref{5.11}.
\end{proof}

\medskip

For \eqref{5.9} we first carry out the integral in $y$ and use \eqref{5.5}.
Then we see that \eqref{5.9} is equal to
\begin{equation}
\begin{split}
&\frac{2\kappa_m}{\sqrt\pi}\sum_{k=0}^{m-1}
\frac{a_k(2m+2k-3)!}{2^{2m+2k-2} t^{m+k-1}}
\int_0^\infty\frac{R_{m+k-2}(u^2)}{u^{2(m+k-1)}}du\\
&=\sqrt\pi \kappa_m \sum_{k=0}^{m-1}
\frac{(-1)^{m+k-1}a_k \varGamma(2m+2k-2)}{2^{2m+2k-2}\varGamma(m+k-1/2)}
\frac1{t^{m+k-1}}.
\end{split}
\label{5.13}
\end{equation}
Since
\[
\frac{\varGamma(2m+2k-2)}{\varGamma(m+k-1/2)}
=\frac{2^{2m+2k-3}\varGamma(m+k-1)}{\sqrt\pi},
\]
which is the direct consequence of \eqref{4.8},
the right hand side of \eqref{5.13} and so \eqref{5.9} are equal to
\[
\frac{\kappa_m}2\sum_{k=0}^{m-1}
\frac{(-1)^{m+k-1}\varGamma(m+k-1)a_k}{t^{m+k-1}}.
\]

Carrying out the integral on $u$ in \eqref{5.10},
we see that \eqref{5.10} is equal to
\[
\frac2{\sqrt\pi}\sum_{k=0}^{m-1}
\frac{(-1)^{m+k-1}(2m+2k-2)!}{2^{2m+2k-1}(m+k-1)!}
\frac1{t^{m+k-1/2}} \int_0^\infty \frac{Q_k(y)}{y^{2m+2k+2}}dy,
\]
which coincides with
\begin{equation}
\frac1\pi\sum_{k=0}^{m-1}\frac{(-1)^{m+k-1}\varGamma(m+k-1/2)}{t^{m+k-1/2}}
\int_0^\infty \frac{Q_k(y)}{y^{2m+2k+2}}dy.
\label{5.14}
\end{equation}
Here we have applied
\begin{equation}
\frac{\varGamma(2z-1)}{\varGamma(z)}
=\frac1{2z-1}\frac{\varGamma(2z)}{\varGamma(z)}
=\frac{2^{2z-2}}{\sqrt\pi}\varGamma\biggl(z-\frac12\biggr).
\label{5.15}
\end{equation}
We should note that the integral in the right hand side of \eqref{5.14}
converges for each integer $k=0,1,2,\dots,m-1$. It is easy to see
that \eqref{4.13} immediately yields that, as $x\to\infty$,
\begin{equation}
Q_k(x)=\kappa_m x^{2m+2k}+o(x^{2m+2k}).
\label{5.16}
\end{equation}
Moreover we deduce from \eqref{5.6} that, as $x\downarrow0$,
\[
Q_k(x)=
\begin{cases}
\kappa_m a_{k+1} x^{2m+2k+2}+O(x^{2m+2k+4})\quad
&\text{if $0\leqq k\leqq m-2$,}\\
\kappa_m a_m x^{4m}\log \dfrac1x+O(x^{4m})\quad
&\text{if $k=m-1$.}
\end{cases}
\]
Hence $Q_k(y)/y^{2m+2k+2}$ is integrable on $(0,\infty)$
for each $k=0,1,2,\dots,m-1$.

For $t>0$ we let
\[
\begin{split}
&P_1(t)=\frac2{\sqrt\pi}\int_1^\infty \frac{Q_{m-1}(y)}{y^3} dy
\int_0^\infty e^{-2\sqrt t uy}R_{2m-2}(u^2)du,\\
&P_2(t)=\frac2{\sqrt\pi}\int_0^1 \frac{Q_m(y)}{y^3} dy
\int_0^\infty e^{-2\sqrt t uy}R_{2m-2}(u^2)du,\\
&P_3(t)=\frac{2\kappa_m a_m}{\sqrt\pi}\int_0^1 y^{4m-3}\log\frac1y \, dy
\int_0^\infty  e^{-2\sqrt t uy}R_{2m-2}(u^2)du,\\
\end{split}
\]
where we put
\[
Q_m(x)=Q_{m-1}(x)-\kappa_m a_m x^{4m}\log\frac1x.
\]
Then \eqref{5.11} is the sum of $P_1(t)$, $P_2(t)$ and $P_3(t)$.

By virtue of \eqref{5.16}, we obtain that $|Q_{m-1}(y)|\leqq C_3 y^{4m-2}$
for $y\geqq1$. Combining this estimate with
$|R_{2m-2}(u^2)|\leqq C_4 u^{4m-2}$, we deduce
\[
|P_1(t)|\leqq C_5 \int_1^\infty y^{4m-5} dy
\int_0^\infty e^{-2\sqrt t uy} u^{4m-2}du
=\frac{C_6}{t^{2m-1/2}}\int_1^\infty\frac{dy}{y^4}.
\]
This means that $P_1(t)$ is of order $1/t^{2m-1/2}$
an is negligible.

We next show that $P_2(t)$ is of order $1/t^{2m-1}$.
It follows from \eqref{5.6} that $|Q_m(y)|\leqq C_7y^{4m}$ for $y<1$.
Noting $R_{2m-2}(x)\leqq0$ for $x\geqq0$, we obtain by the Fubini
theorem that
\[
|P_2(t)|\leqq C_8 \int_0^\infty -R_{2m-2}(u^2)du
\int_0^\infty e^{-2\sqrt t uy} y^{4m-3}dy.
\]
The formula \eqref{5.5} implies that
\[
|P_2(t)|\leqq \frac{C_9}{t^{2m-1}}\int_0^\infty
\frac{-R_{2m-2}(u^2)}{u^{4m-2}}du
=\frac{C_9 \pi}{2\varGamma(2m-1/2)}\frac1{t^{2m-1}}.
\]

The calculation of $P_3(t)$ is slightly complicated but not difficult.
A change of variables from $y$ to $v$ given by $2\sqrt t y=v$ yields
\[
P_3(t)=\frac{\kappa_m a_m}{2^{4m-3}\sqrt\pi t^{2m-1}}
\int_0^{2\sqrt t} v^{4m-3}\log\frac{2\sqrt t}v\, dv
\int_0^\infty e^{-uv} R_{2m-2}(u^2)du.
\]
For $t>1/4$ we set
\[
\begin{split}
&P_3^1(t)=\log(2\sqrt t)\int_0^\infty v^{4m-3} dv
\int_0^\infty e^{-uv}R_{2m-2}(u^2)du,\\
&P_3^2(t)=\log(2\sqrt t)\int_{2\sqrt t}^\infty v^{4m-3} dv
\int_0^\infty e^{-uv}R_{2m-2}(u^2)du,\\
&P_3^3(t)=\int_1^{2\sqrt t} v^{4m-3}\log v \, dv
\int_0^\infty e^{-uv}R_{2m-2}(u^2)du,\\
&P_3^4(t)=\int_0^1 v^{4m-3}\log\frac1v \, dv
\int_0^1 e^{-uv}R_{2m-2}(u^2)du,\\
&P_3^5(t)=\int_0^1 v^{4m-3}\log\frac1v \, dv
\int_1^\infty e^{-uv}R_{2m-2}(u^2)du.\\
\end{split}
\]
Then we have
\[
P_3(t)=\frac{\kappa_m a_m}{2^{4m-3}\sqrt\pi t^{2m-1}}
\{ P_3^1(t)-P_3^2(t)-P_3^3(t)+P_3^4(t)+P_3^5(t)\}.
\]
We show that $P_3^1(t)$ is the leading part of $P_3(t)$
and the others are all negligible. Recall that $R_{2m-2}(u^2)\leqq0$.
It follows from \eqref{5.5} and the Fubini theorem that
\[
P_3^1(t)=\biggl(\frac{\log t}2+\log2\biggr)(4m-3)!
\int_0^\infty \frac{R_{2m-2}(u^2)}{u^{4m-2}}du
=-\frac{\pi \varGamma(4m-2)}{4\varGamma(2m-1/2)}\log t+O(1).
\]
Applying \eqref{4.8} for $z=2m-1$, we have
\[
P_3^1(t)=-2^{4m-5}\sqrt\pi \varGamma(2m-1) \log t +O(1).
\]
The estimate of $P_3^2(t)$ is easy. Indeed, we have
\[
|P_3^2(t)|\leqq C_{10} \log t \int_{2\sqrt t}^\infty
v^{4m-3} dv\int_0^\infty e^{-uv}u^{4m-2}du
\leqq C_{11} \log t \int_{2\sqrt t}^\infty \frac{dv}{v^2},
\]
which is of order $\log t/\sqrt t$. The way of estimates of the remaining
integrals is similar to that of $P_3^2(t)$. We deduce
\[
|P_3^3(t)|\leqq C_{12}\int_1^\infty v^{4m-3}\log v \,dv
\int_0^\infty e^{-uv}u^{4m-2}du
\leqq C_{13} \int_1^\infty \frac{\log v}{v^2}dv\\
\]
and that, by $0\leqq e^{-uv}u^{4m-2}\leqq1$ for $u,v\in[0,1]$,
\[
|P_3^4(t)|\leqq C_{14}\int_0^1 v^{4m-3}\log\frac1v \, dv
\int_0^1 e^{-uv}u^{4m-2}du
\leqq C_{14}\int_0^1 v^{4m-3}\log\frac1v \, dv.
\]
These immediately imply that $P_3^3(t)$ and $P_3^4(t)$ are of order 1.
Note that $R_n(x)$ is asymptotically equal to
$(-1)^{n+1} x^n/n!$ as $x\to\infty$ for an integer $n\geqq0$.
This yields that $|R_{2m-2}(u^2)|$ is bounded by
$C_{15}u^{4m-4}$ for $u\geqq1$ and then we deduce
\[
|P_3^5(t)|\leqq C_{15}\int_0^1 v^{4m-3}\log \frac1v \, dv
\int_1^\infty e^{-uv}u^{4m-4}du
\leqq C_{16}\int_0^1\log\frac1v \, dv=C_{16}.
\]
Therefore, by virtue of Lemma \ref{Lemma 5.4},
we accordingly obtain
\[
\begin{split}
L_2^2(t)=&\frac{\kappa_m}2\sum_{k=0}^{m-1}
\frac{(-1)^{m+k-1}\varGamma(m+k-1)\,a_k}{t^{m+k-1}}\\
&+\frac1\pi \sum_{k=0}^{m-1}
\frac{(-1)^{m+k-1}\varGamma(m+k-1/2)}{t^{m+k-1/2}}
\int_0^\infty \frac{Q_k(y)}{y^{2m+2k+2}}dy\\
&-\frac{\kappa_m a_m \varGamma(2m-1)}4
\frac{\log t}{t^{2m-1}}+O\biggl(\frac1{t^{2m-1}}\biggr),
\end{split}
\]
which implies that we finished to give
the asymptotic behavior of $L_2^0(t)$.

Recall the definition of $\kappa_m$ and $a_m$.
We deduce from \eqref{5.15} that
\[
\begin{split}
-\frac{\kappa_m a_m \varGamma(2m-1)}4&=\frac1{4^{m-1}\varGamma(m)^2}
\frac{(-1)^m}{4^{m-1} m \varGamma(m)^2}\frac{\varGamma(2m-1)}4\\
&=\frac{(-1)^m}{4^m\sqrt\pi\, m \varGamma(m)^3}
\varGamma\biggl(m-\frac12\biggr).
\end{split}
\]
Since $L_2(t)=r^2 L_2^0(t)(t/2r^2)$, we obtain,
by \eqref{5.3} for $M=d-4$,
that
\[
\begin{split}
&-L_1(t)+(-1)^{d/2-1}L_2(t)\\
&=-\sqrt{\frac2\pi}\sum_{n=0}^{d-4}
\frac{(-1)^n(2n)!\,r^{2n+3}\zeta_{2n+3}^{(d)}}{2^n n!}\frac1{t^{n+1/2}}\\
&\hphantom{=}-\sqrt{\frac2\pi}\sum_{n=0}^{d/2-3}
\frac{(-1)^{d/2+n}(2n)!\,r^{2n+3}\varrho_{2n+3}^{(d)}}{2^n n!}
\frac1{t^{n+1/2}}\\
&\hphantom{=}-\kappa_{d/2-1}\sum_{n=0}^{d/2-2}
\frac{(-1)^n\varGamma(d/2+n-2)a_n 2^{d/2+n-3}r^{d+2n-2}\xi_{n}^{(d)}}
{t^{d/2+n-2}}\\
&\hphantom{=}+\frac{\varGamma((d-3)/2)r^{2d-4}}
{\sqrt\pi(d-2)\varGamma(d/2-1)^3}\frac{\log t}{t^{d-3}}
+O\biggl(\frac1{t^{d-3}}\biggr),
\end{split}
\]
where
\[
\xi_k^{(d)}=\int_0^\infty\frac{Q_k(y)}{y^{d+2k}}dy.
\]
It follows from \eqref{5.1} and \eqref{5.2}
that $\zeta_{d-1}^{(d)}=0$ and
\[
\zeta_{2n+3}^{(d)}+(-1)^{d/2}\varrho_{2n+3}^{(d)}=0
\]
for $n=0,1,2,\dots,d/2-3$. Therefore we can conclude that
\[
\begin{split}
-L_1(t)+(-1)^{d/2-1}L_2(t)
=&-\frac{r^{d-2} \xi_0^{(d)}}
{2^{d/2-2}(d-4)\varGamma(d/2-1)}\frac1{t^{d/2-2}}\\
&+\frac1{t^{d/2-1}}\sum_{n=0}^{d-5}
\frac{\alpha_n^{(d)}}{t^{n/2}}\\
&+\frac{\varGamma((d-3)/2)r^{2d-4}}
{\sqrt\pi(d-2)\varGamma(d/2-1)^3}\frac{\log t}{t^{d-3}}
+O\biggl(\frac1{t^{d-3}}\biggr).
\end{split}
\]
Our proof of Theorem \ref{Theorem 5.1} is completed.


\section{Zeros of Macdonald functions}

\noindent
We can find enough properties concerning zeros of $J_\nu$, $Y_\nu$
and $I_\nu$ (cf. \cite{13,19}). However there is
less information on zeros of $K_\nu$.
See \cite{12B} and \cite{21}, for example.

Our purpose in this section is to represent all zeros
of $K_\nu$ as the root of a polynomial of order $N(\nu)$.
Since $K_\nu=K_{-\nu}$ and $N(\nu)\geqq1$ if $|\nu|\geqq3/2$,
it is sufficient to consider the case of $\nu\geqq3/2$. Moreover,
if $\nu=n+1/2$ for an integer $n\geqq1$, the formula
\[
K_\nu(z)=\sqrt{\frac\pi{2z}}e^{-z}\sum_{k=0}^n \frac{(\nu,k)}{(2z)^k}
\]
yields that all zeros of $K_\nu$ are the solutions of the equation
\[
\sum_{k=0}^n \frac{(\nu,n-k)}{2^{n-k}}z^k=0.
\]
Here we have used the notation
\[
(\nu,k)=\frac{\varGamma(\nu+k+1/2)}{k!\varGamma(\nu-k+1/2)}.
\]

From now on, we discuss the cases when $\nu>3/2$ and
$\nu-1/2\notin\mathbb Z$. By virtue of \eqref{2.6},
we have that, for $x>0$
\begin{equation}
\frac{K_{\nu+1}(x)}{K_\nu(x)}=1+\frac{2\nu}x
+\sum_{j=1}^{N(\nu)}\frac1{z_{\nu,j}-x}
+\cos(\pi\nu)\int_0^\infty \frac{dy}{y(y+x)G_{\nu}(y)}.
\label{6.1}
\end{equation}
We can derive the power sum of $z_{\nu,1},z_{\nu,2},\dots,z_{\nu,N(\nu)}$
with the help of \eqref{6.1}. The Newton formula (cf. \cite[p.276]{4}) gives
the polynomial whose roots are $z_{\nu,1},z_{\nu,2},\dots,z_{\nu,N(\nu)}$.

\medskip

\begin{lemma}\label{Lemma 6.1}
Let
\[
p_n=\sum_{j=1}^N z_j^n
\]
for positive integers $n,N$ and $z_1,z_2,\dots,z_N\in\mathbb C$.
We define a sequence $\{\xi_n\}_{n=0}^N$ of complex numbers
by $\xi_0=1$ and
\[
\xi_n=-\frac1n\sum_{k=1}^n \xi_{n-k}p_k
\]
for $n=1,2,\dots,N$. Then we have that, for $z\in\mathbb C$
\begin{equation}
\prod_{j=1}^N (z-z_j)=\sum_{n=0}^N \xi_{N-n}z^n.
\label{6.2}
\end{equation}
\end{lemma}
\begin{proof} Let $s_n$ be the elementary symmetric polynomial
of degree $n$, that is,
\[
s_n=\sum_{1\leqq j_1<j_2<\dots<j_n\leqq N}z_{j_1}z_{j_2}\dots z_{j_n}.
\]
The Newton formula yields that $p_1=-s_1$ and
\[
p_n=\sum_{k=1}^{n-1} (-1)^{n-k+1} s_{n-k}p_k+(-1)^{n+1}n s_n
\]
for $n=2,3,\dots,N$. Therefore we easily deduce (6.2) from the formula
\[
\prod_{j=1}^N (z-z_j)=\sum_{n=0}^N (-1)^{N-k}s_{N-k} z^n.
\]
This completes the proof of this lemma.
\end{proof}

\medskip

We first consider the asymptotic expansion of \eqref{6.1} for large $x$.
Recall that, for $\nu\geqq0$ and  a given integer $M\geqq0$
\[
K_\nu(x)=\sqrt{\frac\pi{2x}}e^{-x}\sum_{n=0}^{M+1}
\frac{(\mu,n)}{(2x)^n}+O\biggl(\frac1{x^{M+2}}\biggr)
\]
as $x\to\infty$ (cf. \cite[p.123]{13}, \cite[p.202]{19})
and we have
\begin{equation}
\frac{K_{\nu+1}(x)}{K_\nu(x)}=\sum_{n=0}^{M+1}\frac{a_n}{x^n}
+O\biggl(\frac1{x^{M+2}}\biggr),
\label{6.3}
\end{equation}
where $\{a_n\}_{n=0}^{M+1}$ is the sequence of real numbers
defined by
\[
\frac{(\nu+1,n)}{2^n}=\sum_{k=0}^n\frac{(\nu,n-k)}{2^{n-k}}a_k
\]
for $n=0,1,2,\dots,M+1$. A simple calculation shows
\[
\begin{split}
&a_0=1,\quad a_1=\nu+\frac12,\quad
a_2=\frac12\biggl(\nu^2-\frac14\biggr),\quad
a_3=-\frac12\biggl(\nu^2-\frac14\biggr),\\
&a_4=-\frac18\biggl(\nu^2-\frac14\biggr)\biggl(\nu^2-\frac{25}4\biggr),\quad
a_5=\frac12\biggl(\nu^2-\frac14\biggr)\biggl(\nu^2-\frac{13}4\biggr).
\end{split}
\]
The remaining constants $a_6, a_7,\dots$ have complicated forms.
It is easy to give the asymptotic expansion of
the right hand side of \eqref{6.1} and then we have that,
for $M\geqq N(\nu)$
\[
\frac{K_{\nu+1}(x)}{K_\nu(x)}=
1+\frac{2\nu}x-\frac1x\sum_{j=1}^{N(\nu)}\frac1{1-z_{\nu,j}/x}
+\frac{\cos(\pi\nu)}x \int_0^\infty \frac{dy}{y(1+y/x)G_\nu(y)}
\]
and
\begin{equation}
\begin{split}
\frac{K_{\nu+1}(x)}{K_\nu(x)}=&1+\frac{2\nu}x-\sum_{n=0}^M
\frac1{x^{n+1}}\sum_{j=1}^{N(\nu)}z_{\nu,j}^n\\
&+\cos(\pi\nu)\sum_{n=0}^M \frac{(-1)^n}{x^{n+1}}
\int_0^\infty \frac{y^{n-1}}{G_\nu(y)}dy+O\biggl(\frac1{x^{M+2}}\biggr).
\end{split}
\label{6.4}
\end{equation}
Here we should note that the integral of $y^{m-1}/G_\nu(y)$
over $(0,\infty)$ converges for each integer $m\geqq0$,
which can be shown by Lemm \ref{Lemma 2.3} and \eqref{2.25}.
Comparing the corresponding coefficients in
\eqref{6.3} and \eqref{6.4}, we obtain
\begin{align}
&N(\nu)=\nu-\frac12+\cos(\pi\nu)\int_0^\infty\frac{dy}{y G_\nu(y)},
\label{6.5}\\
&\sum_{j=1}^{N(\nu)}z_{\nu,j}^n=-a_{n+1}+(-1)^n\cos(\pi\nu)
\int_0^\infty\frac{y^{n-1}}{G_\nu(y)}dy
\label{6.6}
\end{align}
for $n=1,2,\dots,M$. We define a sequence 
$\{\alpha_n^\nu\}_{n=0}^{N(\nu)}$ of complex numbers by
$\alpha_0^\nu=1$ and
\[
\alpha_n^\nu=\frac1n \sum_{k=1}^n \alpha_{n-k}^\nu
\biggl\{ a_{k+1}-(-1)^k \cos(\pi\nu)
\int_0^\infty \frac{y^{k-1}}{G_\nu(y)}dy \biggr\}
\]
for $n=1,2,\dots, N(\nu)$. Therefore, by \eqref{6.6}
and Lemma \ref{Lemma 6.1}, we have the following theorem.

\medskip

\begin{thm}\label{Theorem 6.2}
For $|\nu|>3/2$ the zeros of $K_\nu$
are the solutions of
\[
\sum_{n=0}^{N(\nu)}\alpha_{N(\nu)-n}^\nu z^n=0.
\]
\end{thm}

\medskip

We obtain another polynomial whose roots are
$z_{\nu,1},z_{\nu,2},\dots,z_{\nu,N(\nu)}$
by considering the asymptotic behavior of \eqref{6.1} for small $x$,
which is an improvement of Lemma \ref{Lemma 4.3}.
Let $N\geqq1$ be an integer with $N+1/2<\nu<N+3/2$.
It follows from Lemma \ref{Lemma 4.4} that, as $x\downarrow0$,
\begin{equation}
\frac{K_{\nu+1}(x)}{K_\nu(x)}=\frac{2\nu}x\sum_{n=0}^{2N+1}
b_n x^n+o(x^{2N}),
\label{6.7}
\end{equation}
where $\{b_n\}_{n=0}^{2N+1}$ is a sequence of real numbers defined by
\[
\binom{\nu+1/2}n\frac{2^n\varGamma(2\nu-n+2)}{\varGamma(2\nu+2)}=
\sum_{k=0}^n \binom{\nu-1/2}{n-k}
\frac{2^{n-k}\varGamma(2\nu-n+k)}{\varGamma(2\nu)}b_k
\]
for $n=0,1,2,\dots,2N+1$. A simple calculation shows
\[
\begin{split}
&b_0=1,\quad b_1=b_3=b_5=b_7=0,\quad
b_2=\frac1{4\nu(\nu-1)},\\
&b_4=-\frac1{16\nu(\nu-1)^2(\nu-2)},\quad
b_6=\frac1{32\nu(\nu-1)^3(\nu-2)(\nu-3)}.
\end{split}
\]
The remaining coefficients have complicated forms.
We easily get
\begin{equation}
\begin{split}
\frac{K_{\nu+1}(x)}{K_\nu(x)}=&
\frac{2\nu}x+1+\sum_{n=0}^{2N} x^n 
\sum_{j=1}^{N(\nu)}\frac1{z_{\nu,j}^{n+1}}\\
&+\cos(\pi\nu)\sum_{n=0}^{2N} (-1)^n x^n
\int_0^\infty \frac{dy}{y^{n+2}G_\nu(y)}+o(x^{2N}).
\end{split}
\label{6.8}
\end{equation}
We have to remark that we can not derive the higher term in \eqref{6.8}
while $M$ in \eqref{6.4} is arbitrary. This deference
is caused from the integrabiliy of $1/y^m G_\nu(y)$.
Comparing the corresponding coefficients in
\eqref{6.7} and \eqref{6.8}, we obtain
\begin{align}
&\sum_{j=1}^{N(\nu)}\frac1{z_{\nu,j}}=-1-\cos(\pi\nu)
\int_0^\infty \frac{dy}{y^2 G_\nu(y)},
\label{6.9}\\
&\sum_{j=1}^{N(\nu)}\frac1{z_{\nu,j}^n}=2\nu b_n+(-1)^n\cos(\pi\nu)
\int_0^\infty \frac{dy}{y^{n+1} G_\nu(y)}
\label{6.10}
\end{align}
for $n=2,3,\dots,2N$.  We define a sequence 
$\{\beta_n^\nu\}_{n=0}^{N(\nu)}$ of complex numbers by
$\beta_0^\nu=1$,
\[
\begin{split}
&\beta_1^\nu=1+\cos(\pi\nu)
\int_0^\infty \frac{dy}{y^2 G_\nu(y)},\\
&\beta_n^\nu=-\frac1n\sum_{k=2}^n \beta_{n-k}^\nu
\biggl\{2\nu b_k+(-1)^k\cos(\pi\nu)
\int_0^\infty \frac{dy}{y^{k+1} G_\nu(y)}\biggr\}\\
&\hphantom{\beta_n^\nu=}+\frac1n \beta_{n-1}^\nu
\biggl\{1+\cos(\pi\nu)\int_0^\infty \frac{dy}{y^2 G_\nu(y)}\biggr\}
\end{split}
\]
for $n=2,3,\dots,N(\nu)$. By \eqref{6.9}, \eqref{6.10}
and Lemma \ref{Lemma 6.1}, we have the following theorem.

\medskip

\begin{thm}\label{Theorem 6.3}
For $|\nu|>3/2$ the zeros of $K_\nu$ are the solutions of
\[
\sum_{n=0}^{N(\nu)}\beta_n^\nu z^n=0.
\]
\end{thm}


\begin{remark}\label{Remark 6.4}
It is known that, if $\nu-1/2$ is not an odd integer,
\[
N(\nu)=\nu-\frac12+\frac{\theta_\nu}\pi
\]
(cf. \cite[p.512]{19}). Here $\theta_\nu$ is the unique
number determined by
\[
|\theta_\nu|<\pi,\quad \cos\theta_\nu=\sin\pi\nu,\quad
\sin\theta_\nu=\cos\pi\nu.
\]
Hence we deduce from \eqref{6.5} that,
if $\nu>3/2$ and $\nu-1/2\notin\mathbb Z$,
\begin{equation}
\int_0^\infty \frac{dy}{yG_\nu(y)}
=\frac{\theta_\nu}{\pi\cos(\pi\nu)}.
\label{6.11}
\end{equation}
Similarly, when $0\leqq\nu<3/2$ and $\nu\neq1/2$, we can easily
derive \eqref{6.11} by virtue of \eqref{2.5} and \eqref{6.3}.
This implies that, for an integer $n\geqq0$,
\[
\int_0^\infty \frac{dy}{y\{K_n(y)^2+\pi^2 I_n(y)^2\}}
=\frac12,
\]
which has been obtained in \cite{21}.
\end{remark}

\medskip

\begin{remark}\label{Remark 6.5}
When $3/2<\nu<7/2$, $N(\nu)=2$ and the zeros of $K_\nu$
satisfy some quadratic equations.
In particular, when $\nu=2$, the zeros $z_{2,1}$ and $z_{2,2}$
satisfy
\[
\begin{cases}
\displaystyle
z_{2,1}+z_{2,2}=-\frac{15}8-\int_0^\infty\frac{dy}{G_2(y)},\\
\quad\\
\displaystyle
z_{2,1}^2+z_{2,2}^2=\frac{15}8+\int_0^\infty\frac{ydy}{G_2(y)},
\end{cases}
\quad
\begin{cases}
\displaystyle
\frac1{z_{2,1}}+\frac1{z_{2,2}}=-1-\int_0^\infty\frac{dy}{y^2G_2(y)},\\
\quad\\
\displaystyle
\frac1{z_{2,1}^2}+\frac1{z_{2,2}^2}
=\frac12+\int_0^\infty\frac{dy}{y^3G_2(y)}.\\
\end{cases}
\]
By using Mathematica we obtain from each system of equations
that the zeros are close to $-1.28 \pm 0.43\, i$, and check the comment
\lq\lq The two zeros of $K_2(z)$ are not very far from the points
$-1.29\pm 0.44\, i$.\rq\rq in \cite[p.512]{19}.
Moreover we find that the zeros of $K_3$ are
close to $-1.68\pm 1.31\, i$ in the similar way.
\end{remark}


\noindent Y. Hamana\\
Department of Mathematics\\
Kumamoto University\\
Kurokami 2-39-1\\
Kumamoto 860-8555\\
Japan\\
e-mail: hamana(at)kumamoto-u.ac.jp

\bigskip

\noindent
H. Matsumoto\\
Department of Physics and Mathematics\\
Aoyama Gakuin University\\
Fuchinobe 5-10-1\\
Sagamihara 252-5258\\
Japan\\
e-mail: matsu(at)gem.aoyama.ac.jp


\begin{thebibliography}{99}

\bibitem{1} R. K. Getoor,
`Some asymptotic formulas involving capacity', 
{\em Z. Wahr. Verw. Gebiete }4 (1965) 248--252.


\bibitem{2}
R. K. Getoor and M. J. Sharpe,
`Excursions of Brownian motion and Bessel processes',
{\em Z. Wahr. Ver. Gebiete }47 (1979) 83--106.

\bibitem{3}
I. S. Gradshteyn and I. M. Ryzhik,
Table of Integrals, Series, and Products, 7th ed.
(Academic Press, Amsterdam 2007).

\bibitem{4}
L. C. Grove,
Algebra (Academic Press, New York 1983).

\bibitem{5}
Y. Hamana,
`On the expected volume of the Wiener sausage',
{\em J. Math. Soc. Japan} 62 (2010) 1113--1136.

\bibitem{6}
Y. Hamana,
`The expected volume and surface area of the Wiener sausage
in odd dimensions',
{\em Osaka J. Math.} 49 (2012) 853--868.


\bibitem{7}
Y. Hamana and H. Matsumoto,
`The probability distributions of the first hitting times
of Bessel processes',
{\em Trans. Amer. Math. Soc.} arXiv:1106.6132 [math.PR]
(toappear)

\bibitem{8}
Y. Hamana and H. Matsumoto,
`The probability densities of the first hitting times
of Bessel processes',
{\em J. Math-for-Ind.} 4B (2012) 91--95.

\bibitem{9}
M. G. H. Ismail,
`Integral representations and complete monotonicity
of various quotients of Bessel functions',
{\em Canad. J. Math.} 29 (1977) 1198--1207.

\bibitem{10}
M. G. H. Ismail and D. H. Kelker,
`Special functions, Stieltjes transforms and infinite divisibility',
{\em SIAM J. Math. Anal.} 10 (1989) 884--901.

\bibitem{11}
K. It{\^o} and H. P. McKean Jr.,
Diffusion Processes and Their Sample Paths
(Springer-Verlag, Berlin-New York 1974).

\bibitem{12}
J. T. Kent,
`Some probabilistic properties of Bessel functions',
{\em Ann. Probab.} 6 (1978) 760--770.

\bibitem{12B}
M. K. Kerimov and S. L. Skorokhodov,
`Calculation of complex zeros of a modified Bessel function of
the second kind and its derivatives',
{\em U. S. S. R. Comput. Math. and Math. Phys.} 24 (1984) 115--123
(Russian original, {\em  Zh. Vychisl. Mat. i Mat. Fiz.} 24 (1984) 1150--1163).


\bibitem{13}
N. N. Lebedev,
Special Functions and Their Applications
(Dover, New York 1972).

\bibitem{14}
J. -F. Le~Gall,
`Sur une conjecture de M.~Kac',
{\em Probab. Th. Rel. Fields} 78 (1988) 389--402.

\bibitem{15}
J. -F. Le~Gall,
`Wiener sausage and self-intersection local times',
{\em J. Funct. Anal.} 88 (1990) 299--341.

\bibitem{16}
S. C. Port,
`Asymptotic expansions for the expected volume of a stable sausage',
{\em Ann. Probab.} 18 (1990) 492--523.

\bibitem{17}
D. Revuz and M. Yor,
Continuous Martingales and Brownian Motion, 3rd~ed.
(Springer-Verlag, Berlin 1999).

\bibitem{18}
F. Spitzer,
`Electrostatic capacity, heat flow and Brownian motion',
{\em Z. Wahr. Verw. Gebiete} 3 (1964) 110--121. (1964)

\bibitem{19}
G. N. Watson,
A Treatise on the Theory of Bessel Functions, Reprinted of 2nd ed.
(Cambridge University Press, Cambridge 1995)

\bibitem{20}
M. Yamazato,
`Hitting time distributions of single points
for 1-dimensional generalized diffusion processes',
{\em Nagoya Math. J.} 119 (1990) 143--172.


\bibitem{21}
M. V. Zavolzhenskii and A. Kh. Terskov,
`The zeros of the cylinder functions $K_n(z)$',
{\em U. S. S. R. Comput. Math. and Math. Phys.} 17 (1978) 192--195
(Russian original, {\em  Zh. Vychisl. Mat. i Mat. Fiz.} 17 (1987) 759--762).

\end{thebibliography}
\end{document}